\documentclass{article}
\usepackage[utf8]{inputenc}
\usepackage[english]{babel}
\usepackage{latexsym}
\usepackage{amsfonts}
\usepackage{soul}
\usepackage{mathrsfs}
\usepackage{amsthm}
\usepackage[all]{xy}
\usepackage{mathtools}
\usepackage{color}
\usepackage{amssymb, mathrsfs, amsfonts, amsmath}
\usepackage{mathtools}
\usepackage{graphicx}
\usepackage[font=small,labelfont=bf]{caption}
\usepackage{epstopdf}
\usepackage[pdfpagelabels,hyperindex]{hyperref}
\usepackage{xcolor}
\usepackage{float}
\usepackage{pgfplots}
\usepackage{listings}
\usepackage{longtable}
\usepackage{mathrsfs}
\usepackage{amsbsy}
\usepackage{tikz}
\usetikzlibrary{calc}
\usepackage{array,tabularx,float, multicol,multirow,graphicx}
\usetikzlibrary{calc,intersections,arrows.meta,through,patterns}
\usetikzlibrary{decorations.pathmorphing}
\usepackage{pgfplots}
\pgfplotsset{compat=1.18}
\usetikzlibrary{calc,through,intersections,patterns,patterns.meta}
\usepackage{wasysym}
\usepackage{stackrel}
\pgfplotsset{compat=newest}
\usetikzlibrary{calc,through,intersections,patterns,patterns.meta}
\usepackage{pgfplots}
\usepgfplotslibrary{groupplots}
\pgfplotsset{compat=1.18}
\usepackage{enumitem}

\newtheorem{theorem}{Theorem}[section]
\newtheorem{thm}{Theorem}[section]
 \newtheorem{cor}[thm]{Corollary} 
 \newtheorem{lem}[thm]{Lemma}
 \newtheorem{remark}[thm]{Remark}
 \newtheorem{prop}[thm]{Proposition}

 \newtheorem{definition}[thm]{Definition}
 \newcommand{\grad}{\nabla\,}
 \newcommand{\diver}{{{\rm{div}}}}
 
\usepackage{graphicx} 
\usepackage{eucal}
\usepgfplotslibrary{groupplots}

\usepackage[english]{babel}

\usepackage[letterpaper,top=2cm,bottom=2cm,left=3cm,right=3cm,marginparwidth=1.75cm]{geometry}

\usepackage{amsmath}
\usepackage{graphicx}

\providecommand{\keywords}[1]{%
  \small\textbf{Keywords ---} #1
}

\title{Barta Theorem for the $p$-Laplacian \\and Geometric Applications}
\author{Paulo Henryque C. Silva \thanks{Universidade Federal do Ceará. Departamento de Matemática. Fortaleza - CE, Brazil. \texttt{paulohenryque@alu.ufc.br}}%
}

\begin{document}
\maketitle
    
\begin{abstract}
\noindent In this article, we develop a Barta-type formulation for the $p$-Laplacian on Riemannian manifolds, extending the approach of Cheung-Leung \cite{cheung} and  Bessa--Montenegro \cite{BessaMontenegro}  from the linear to the nonlinear setting. This framework yields sharp lower bounds for the $p$-fundamental tone without any assumptions on boundary regularity. As applications, we obtain nonlinear extensions of Cheng’s eigenvalue comparison theorem and  the Cheng--Li--Yau estimate for $p \geq 2$ in the context of minimal immersions. In particular, under the above assumptions, the domain $\Omega$ is $p$-stable for the Schr\"odinger-type operator associated with the potential $\mathcal{V} = ||A||^{p}$, where $A$ denotes the second fundamental form of the minimal immersion. In addition, we establish a lower bound for the $p$-fundamental tone in the setting where the immersion has locally bounded mean curvature. Finally, we provide a Kazdan--Kramer type characterization of the $p$-fundamental tone, offering a unified and geometric perspective on spectral bounds for the operator $p$-Laplacian.
\end{abstract}



\section{Introduction} 
Quasilinear operators play a fundamental role in both mathematical analysis
and physical modeling, due to their broad range of applications. They naturally
arise in the description of several nonlinear phenomena, including non-Newtonian
flows, nonlinear diffusion processes, and saturation effects. In this work,
we investigate the quasilinear operator known as the \emph{$p$-Laplacian}, on a Riemannian $m$-manifold $(M,g)$, defined by
\begin{equation}
    \triangle_p \varphi \coloneq \operatorname{div}(|\nabla \varphi |^{p-2}\nabla \varphi),
\end{equation}
with $p\in(1,\infty)$. In the case $p=2$, we recover the classical Laplace--Beltrami operator on $M$. 
When $p>2$, the $p$-Laplacian is associated with \emph{slow diffusion} phenomena, in which the diffusivity increases with the magnitude of the gradient $|\nabla \varphi|$. For $p \in (1,2)$, the $p$-Laplacian describes \emph{fast diffusion} processes, where the diffusivity exhibits an inverse dependence on $|\nabla \varphi|$. Standard references for the foundational study of the
$p$-Laplacian include \cite{anane, le, lindqvist, lindqvist92}.

Let $(M,g)$ be a Riemannian manifold and $\Omega \subset M$ be a bounded domain, possibly with boundary $\partial \Omega$. The \emph{eigenvalue problem} associated with the
$p$-Laplacian $\triangle_p$ is given by
\begin{equation}\label{Dirichleteigenvalue}
\begin{cases}
\triangle_p \varphi = -\lambda\,|\varphi|^{p-2}\varphi 
& \text{in} \ \Omega,\\
\varphi = 0 
& \text{on} \ \partial\Omega, \quad \text{if} \quad \partial \Omega \neq \emptyset.
\end{cases}
\end{equation}
The real number $\lambda$ is called an \emph{eigenvalue of the $p$-Laplacian} $\triangle_p$ on $M$. For $p \in (1,\infty)$, it follows from \cite{lindqvist} that the Dirichlet problem
\eqref{Dirichleteigenvalue} admits a non-trivial solution, and from \cite[Theorem 4.5]{le} that any such
solution has regularity only of class $\mathcal{C}^{1,\alpha}(\Omega)$ for some
$\alpha \in (0,1)$. Therefore, it is natural to interpret solutions to \eqref{Dirichleteigenvalue} in the weak sense. More precisely, we say that $\varphi \in W^{1,p}_0(\Omega)$ is a \emph{weak solution} to \eqref{Dirichleteigenvalue} if
\begin{equation}\label{eqweak}
\int_{\Omega} \left\langle |\grad\varphi |^{p-2} \grad \varphi, \grad \eta \right\rangle d\mu = \lambda \int_{\Omega} | \varphi |^{p-2} \varphi \eta \, d\mu,  
\end{equation}
for any test function $\eta \in \mathcal{C}_0^{\infty}(\Omega)$. Here we are writing the metric as $g = \langle \cdot , \cdot \rangle$ and  $d\mu$ denotes the Riemannian volume measure induced by the metric $g$ on $\Omega$. The Sobolev space $\mathrm{W}^{1,p}_0(\Omega)$ denotes the completion of $\mathcal{C}_0^{\infty}(\Omega)$ with respect to the norm 
\begin{equation*}
    \| \varphi \|_{1,p} = \left( \int_{\Omega} \left( |\varphi|^p + |\nabla \varphi|^p \right) d\mu \right)^{1/p}.
\end{equation*}
Let $\lambda_{1,p}^{\mathcal{D}}(\Omega)$ be the infimum of all values $\lambda$ for which \eqref{Dirichleteigenvalue} admits a nontrivial solution, whenever $\partial \Omega \neq \emptyset$. The Dirichlet spectrum of the $p$-Laplacian consists of a nondecreasing sequence $\{\lambda_{k,p}(\Omega)\}_{k \in \mathbb{N}}$, with $\lambda_{k,p}(\Omega) \to +\infty$ as $k \to \infty$, obtained via the Ljusternik--Schnirelman principle, see   \cite{garzia-peral}. Moreover, the first eigenvalue is isolated,  simple,   admits an eigenfunction of fixed sign, and possesses a variational characterization (see \cite{anane, lindqvist}),
\begin{equation}\label{eq:Rayleigh quotient}
    \lambda_{1,p}^{\mathcal{D}}(\Omega)
    = \inf \left\{ \frac{\displaystyle \int_{\Omega} |\nabla \varphi |^p d\mu} {\displaystyle \int_{\Omega} |\varphi|^p d\mu}; \, \varphi \in \mathrm{W}^{1,p}_0(\Omega), \, \varphi \not\equiv 0
    \right\}.
\end{equation} Clearly,  \eqref{eq:Rayleigh quotient} is valid for  arbitrary open subsets $\Omega \subseteq M$ and defines what is called the \emph{$p$-fundamental tone} $\lambda_{1,p}(\Omega)$ of  $\Omega $, although $\Omega$ may not have Dirichlet eigenvalues.
When $\Omega \subset M$ has a nonempty and relatively compact boundary, the $p$-fundamental tone coincides with the first Dirichlet eigenvalue,  $\lambda_{1,p}(\Omega) = \lambda_{1,p}^{\mathcal{D}}(\Omega).$ Several significant results involving the $p$-fundamental tone have been obtained in recent years. For the case $p=2$, we refer to \cite{BessaMontenegro, BessaMontenegro2003}, while for $p \in (2,\infty)$ we mention \cite{LimaMontenegroSantos2010, Mao2014, matei}. A rather interesting way to think about equation \eqref{Dirichleteigenvalue},  $p=2$,  is by recalling the problem of \textit{vibrating drums}: let $\mathcal{V} \colon \Omega \times \mathbb{R}_+ \to \mathbb{R}$ describe the vibration of a drumhead satisfying the wave equation
\begin{equation}\label{SoundEquation}
\begin{cases}
\mathcal{V}_{tt} = \alpha^2 \triangle \mathcal{V}
& \text{in } \Omega \times \mathbb{R}_+,\\
\mathcal{V} = 0
& \text{on } \partial\Omega.
\end{cases}
\end{equation}
Here, the constant $\alpha$ depends on the physical characteristics of the membrane. 
Looking for solutions of the form $\mathcal{V}(x,t) = \mathcal{T}(t)\vartheta(x)$, we obtain an ordinary Sturm--Liouville equation for $\mathcal{T}(t)$ and a Dirichlet eigenvalue problem \eqref{Dirichleteigenvalue} for the function $\vartheta(x)$. 
In this context, $\lambda_{1}^{\mathcal{D}}(\Omega)$ represents exactly the lowest vibration frequency of the membrane (the deepest sound emitted by the drum). 
The following theorem then addresses a bound on the deepest sound emitted by the drum through an admissible test function.

\begin{theorem}[Barta's inequality \cite{Barta1937}]\label{thm:barta}
Let \(\Omega\) be a bounded domain in a  Riemannian $m$-manifold $M$ with piecewise smooth non-empty boundary \(\partial \Omega\), and let \(\eta \in \mathcal{C}^{2}(\Omega)\cap \mathcal{C}^{0}(\overline{\Omega})\) satisfy \(\eta>0\) in \(\Omega\) and \(\eta|_{\partial \Omega}=0\). Let \(\lambda_{1}^{\mathcal{D}}(\Omega)\) denote the first Dirichlet eigenvalue of \(M\). Then
\begin{equation}\label{eq:barta}
\inf_{\Omega}\!\left\{-\frac{\triangle \eta}{\eta} \right\} \leq
\lambda_{1}^{\mathcal{D}}(\Omega) \leq  \sup_{\Omega}\!\left\{-\frac{\triangle \eta}{\eta} \right\}. 
\end{equation}
Moreover, equality holds if and only if \(\eta\) is a first eigenfunction of \(\Omega\).
\end{theorem}
Figure \eqref{figbarta} illustrates the importance of selecting appropriate test functions $\eta$ in Theorem \eqref{thm:barta} to obtain effective control of the first Dirichlet eigenvalue of the Laplacian

\begin{figure}[htb!]
		\centering
		\def\r{7}
		\begin{tikzpicture}[scale=2.2,>=stealth]
			\def\u{1}
			\def\v{1}
			
			\path[draw,name path=l1] (-0,0) -- (2,0);
			\path[draw,name path=l2] (0,0) -- (0,1.5*\u);
			\path[draw,name path=l3,dashed,very thin] (2,0) -- (2,1.5*\u);
			\path[draw,domain=0.13:1.88,blue,name path=para,samples=500] plot ({\x},{-0.3+2.38*pow(\x-1,2)});
			
			\draw[] (0,0) to[out=0,in=120] (0.5,0.6) to[out=-60,in=180] (0.8,0.3) to[out=0,in=180] (1.3,0.9) to[out=0,in=180] (2,0);
			\fill (2,0) circle (0.5pt);
			
    			\node (t1) at (0.42, 1.38) {{\scriptsize $-\frac{\triangle \eta}{\eta}(y)$}};
			\node (t2) at (1.3,1) {{\scriptsize $\eta(y)$}};
			
			\begin{scope}[xshift=2.5cm]
				\path[draw,name path=l1] (-0,0) -- (2,0);
				\path[draw,name path=l2] (0,0) -- (0,1.5*\u);
				\path[draw,name path=l3,dashed,very thin] (2,0) -- (2,1.5*\u);
				\path[draw,domain=0.1:1.88,blue,name path=para,samples=500] plot ({\x},{0.7+1.0*pow(\x-1,2)});

				\path[draw,domain=0:2,name path=para1,samples=500,y=0.4cm] plot ({\x},{\x*(2-\x)});
				
				\node (t3) at (0.44,1.38) {{\scriptsize $-\frac{\triangle \eta}{\eta}(y)$}};
				\node (t4) at (1,0.48) {{\scriptsize $\eta(y)$}};
			\end{scope}
			
			\begin{scope}[xshift=5cm]
				\path[draw,name path=l1] (-0,0) -- (2,0);
				\path[draw,name path=l2] (0,0) -- (0,1.5*\u);
				\path[draw,name path=l3,dashed,very thin] (2,0) -- (2,1.5*\u);
				\path[draw,domain=0:2,blue,name path=para,samples=500] plot ({\x},{1.5});
				
				\path[draw,domain=0:2,name path=para1,samples=500,y=0.6cm] plot ({\x},{\x*(2-\x)});
				
				\node (t5) at (1.0,1.38) {{\scriptsize $-\frac{\triangle \eta}{\eta}(y)$}};
				\node (t6) at (1,0.7) {{\scriptsize $\eta(y)$}};
				\fill[blue] (0,1.5) circle (0.5pt);
			\end{scope}

		\end{tikzpicture}
    \caption{Examples of admissible test functions on the interval $[0,1]$.}
    \label{figbarta}
\end{figure}
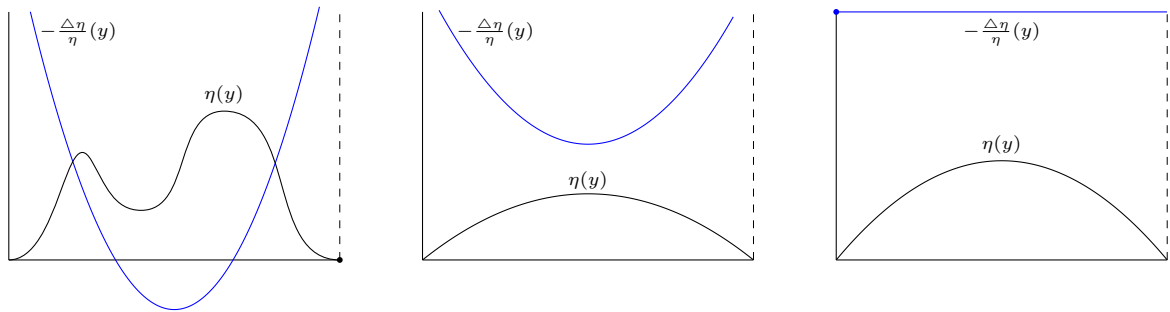

The lower inequality \eqref{eq:barta} above admits a natural extension to arbitrary  domains $\Omega \subseteq M$, as established by Bessa--Montenegro in \cite{BessaMontenegro}. Moreover, a version of Theorem \ref{thm:barta} for the $p$-Laplacian on domains $\Sigma \subset \mathbb{R}^n$ can be found in \cite[Section 2.2] {AllegrettoHuang1998}. In this manuscript, adapting the techniques developed in
\cite{AllegrettoHuang1998, BessaMontenegro}, we establish an analog of \cite[Proposition 2.1]{BessaMontenegro} for the $p$-fundamental tone. Throughout the following theorems concerning the $p$-Laplacian operator, we assume that $p \in (1,\infty)$, unless otherwise stated.

\begin{theorem}[$p$-Barta's inequality]\label{thm:pbarta} Let $(M,g)$ be a smooth Riemannian manifold, and let $\Omega \subset M$ be a domain. Assume that $\eta \in \mathcal{C}^{1+\alpha}(\Omega)\cap\mathcal{C}^{0}(\overline{\Omega})$ with $\triangle_p \eta \in \mathcal{C}^0(\Omega)$
and $\eta>0$ in $\Omega$.
Then the $p$-fundamental tone $\lambda_{1,p}(\Omega)$ satisfies
\begin{equation}\label{eq:pBarta}
\lambda_{1,p}(\Omega)
\ge
\inf_{\Omega}
\left\{
-\frac{\triangle_p \eta}{\eta^{p-1}}
\right\}.
\end{equation}
Moreover, if $\Omega$ is bounded, the equality in \eqref{eq:pBarta}
holds if and only if $\eta=\beta u$ for some constant $\beta>0$,
where $u \in W^{1,p}_0(\Omega)$ is a nonnegative minimizer of the Rayleigh quotient.
\end{theorem}

Another elegant application of Theorem \eqref{thm:barta} is the classical theorem of Cheng \cite{Cheng1975}, which establishes an important comparison result for Dirichlet eigenvalues (see also \cite[Theorem 1.3]{BessaMontenegro}).

\begin{theorem}[Cheng \cite{Cheng1975}] \label{thm:cheng} Let $M$ be an $m$-dimensional Riemannian manifold and let $B_M(p,r)$ be the geodesic ball centered at $p_0$ with radius $r<\operatorname{inj}(p_0)$.  
Assume that the sectional curvatures of $M$ on $B_M(p_0,r)$ are bounded above by $c$, and let  $\mathbb{M}^m(c)$ denote the simply connected $m$-dimensional space form of constant sectional curvature $c$. 
Then,
\begin{equation}
\lambda_1^{\mathcal{D}} \big( B_M(p_0,r) \big) \ge \lambda_1^{\mathcal{D}} \big(B(r) \big).
\end{equation}
Where $B(r)$ is the ball of radius $r$ in the model space $\mathbb{M}^m(c)$.
\end{theorem}

Through Theorem \ref{thm:pbarta}, we obtain geometric applications for the $p$-Laplacian that naturally extend the results established by Bessa--Montenegro in \cite[Section 2]{BessaMontenegro}. Theorem \ref{thm:cheng} was originally extended to arbitrary domains in the work \cite{BessaMontenegro}. The following theorem can be regarded as a $p$-Laplacian analog of Cheng's theorem. Its proof follows the strategy introduced by Bessa--Montenegro \cite[Theorem 1.8]{BessaMontenegro}, combined with the arguments developed by Lima--Montenegro--Santos \cite{LimaMontenegroSantos2010}. Additional comparison results can be found in \cite{BarrosoBessa2006, BessaMontenegro2003, matei, Mao2014}.

\begin{theorem}\label{thm:Cheng}
    Let $(M^m,g)$ be a Riemannian manifold, and let $p_0\in N$. Assume that the radial sectional curvature of $M$ satisfies $K(\partial_t,v) \leq c$, $x\in B_M(p_0,r)\setminus \operatorname{Cut}(p_0)$, $v\perp \partial_t,\ |v|\leq1$, where $\partial_t$ denotes the unit vector radial field along the minimal geodesics issued from $p_0$. Let $\mathbb M^m(c)$ be the simply connected $n$-dimensional space form of constant sectional curvature $c$, and suppose that $\mathcal{H}^{m-1}\big(\operatorname{Cut}(p_0)\cap B_M(p_0,r)\big)=0$. Then the bottom of the Dirichlet spectrum of the geodesic ball $B_M(p_0,r)$ satisfies
\begin{equation}
\lambda_{1,p}\big(B_M(p_0,r)\big) \geq \lambda_{1,p}^{\mathcal{D}}\big(B(r)\big).    
\end{equation}\label{comparison}
Moreover, equality holds if and only if the geodesic balls
$B_M(p_0,r)$ and $B(r)$ are isometric.
\end{theorem}

Another result in the same spirit as Theorem \ref{thm:cheng} is the classical
Cheng--Li--Yau Theorem (see \cite{ChengLiYau1984}), which provides a lower bound
for the first Dirichlet eigenvalue of a minimal immersion $M^{m} \hookrightarrow \mathbb{N}^{n}(c)$.

\begin{theorem}[Cheng--Li--Yau]\label{thm:ChengLiY}
Let $M^{m}\subset \mathbb{M}^{n}(c)$ be an immersed $m$-dimensional minimal submanifold
of the $n$-dimensional space form of constant sectional curvature
$c\in\{-1,0,1\}$, and let $K\subset M^{m}$ be a compact domain with $\mathcal{C}^{2}$ boundary.
Let
\[
r:=\inf_{p_0\in K}\,\sup_{z\in K}\,
\operatorname{dist}_{\mathbb{M}^{n}(c)}(p_0,z)>0,
\]
be the outer radius of $K$.
If $c=1$, assume additionally that $r \leq \pi/2$.
Then
\begin{equation}\label{eq:ChengLiYau}
    \lambda_{1}^{\mathcal{D}}(K)\ge \lambda_{1}^{\mathcal{D}}\big(B(r)\big).
\end{equation}
Moreover, the equality in \eqref{eq:ChengLiYau} holds if and only if
$M$ is totally geodesic in $\mathbb{M}^{n}(c)$ and
$K=B(r)$.
\end{theorem}

The above theorem admits a natural extension to the fundamental tone, as established by \cite{BessaMontenegro}. The following result extends this theorem to the nonlinear setting of the $p$-Laplacian operator. Due to the $p$-homogeneity of the operator $\triangle_p$, a restriction arises on the radius $r_\star(c)$ of the ball in the $m$-dimensional model space $\mathbb{M}^m(c)$ of constant sectional curvature $c$ where the comparison holds. In Theorem 1.10 of \cite{BessaMontenegro}, the restriction on $r_\star(c)$ appears naturally when the model space is spherical, i. e, $c>0$; in this case, one has $r_\star(1) < \pi/2$, while for $c \leq 0$ there is no dependence on the curvature of the space form, that is, $r_\star(c)=r$. The theorem stated below provides a natural extension of \cite[Theorem 1.10]{BessaMontenegro} to the range $p\in[2,\infty)$. In particular, we are able to preserve the comparison in the case $c>0$ for all $r\in(0,\pi/2)$. However, when $c=0$ or $c<0$, the dependence on the model space $\mathbb{M}^m(c)$ induced by the nonlinearity of the operator cannot be removed, leading to geometric obstructions to the comparison beyond a certain critical radius $r_\star(c)<r$. Additional results in the literature providing bounds for the $p$-fundamental tone can be found in \cite{BarrosoBessa2006, DuMao2017, matei, PHCsilvapPinch}.

\begin{theorem}\label{thm:pminsub_ball}
Let $(N^n,g)$ be an $n$-dimensional Riemannian manifold and fix $p_0\in N$. Assume that the radial sectional curvatures of $N$ with respect to $p_0$ satisfy $K_N(\partial_t,v) \leq c$ for every $x\in B_N(p_0,r)\setminus \mathrm{Cut}(p_0)$ and every unit vector 
$v\perp \partial_t$. Suppose moreover that $\mathcal{H}^{n-1}\!\big(\mathrm{Cut}(p_0)\cap B_N(p_0,r)\big)=0$. Let $\varphi:M^m\hookrightarrow N^n$ be a minimal immersion and let 
$\Omega\subset \varphi^{-1}(B_N(p_0,r))$ be a connected component. 
Assume that the extrinsic distance function $t(x)=\operatorname{dist}_N(\varphi(x),p_0)$ satisfies $|\nabla^M t|\ge k>0$ in $\Omega$. Let $B(r)\subset \mathbb{M}^m(c)$ denote the geodesic ball of radius $r$ in the 
$m$-dimensional simply connected space form of constant curvature $c$. 
Then for every $2\le p<\infty$ and $r\le r_\ast(c)$, it holds that
\[
\lambda_{1,p}(\Omega)
\ge
k^{p-2} \cdot \lambda_{1,p}^{\mathcal D}(B(r)).
\]
If equality holds for some bounded $\Omega$, then necessarily 
$k=1$ and $\Omega$ is isometric to $B(r)$.
\end{theorem}

\begin{remark}
The condition $|\nabla^{M}t(x)| \ge k >0$ arises naturally in the study of complete submanifolds with controlled extrinsic geometry. In \cite{BessaBarbosaMontenegro2007}, the authors prove that if a complete 
submanifold $M^m \subset \mathbb{R}^n$ satisfies $t(x)|\alpha(x)| < 1$ outside a compact set, where $t(x)=\operatorname{dist}_{\mathbb{R}^n}(0,\varphi(x))$, then the extrinsic distance function has no critical points at infinity. In particular, $|\nabla^M t|$ does not vanish outside a compact set. More generally, the notion of asymptotically extrinsically tamed submanifolds introduced in \cite[Theorem 1.1]{BessaGimenoPalmer2017} provides a geometric framework ensuring that the radial direction does not become tangent to $M$ at infinity. Consequently, the angle function $\cos{\alpha(x)}=\langle \nabla^N t, \nu\rangle$ remains non-degenerate, which guarantees the validity of the hypothesis imposed in Theorem \ref{thm:pminsub_ball} for a larger class of minimal immersed submanifolds.
\end{remark}

\begin{remark}
When $p=2$, the constant $k^{p-2}$ equals $1$, and we do not need  the condition $|\nabla^{M}t(x)| \ge k >0$ and Theorem \eqref{thm:pminsub_ball} reduces precisely to the comparison result of \cite[Theorem 1.10]{BessaMontenegro}.
\end{remark}

The notion of stability plays a central role in the theory of minimal hypersurfaces.
In the classical case $p=2$, a domain $D$ in a minimal hypersurface $M \subset \mathbb{R}^{n+1}$ is said to be stable if the second variation of area 
is nonnegative or, equivalently, if the first Dirichlet eigenvalue of the 
Jacobi operator $L=\triangle + \|A\|^2$ is nonnegative. 
This condition has deep geometric consequences. In particular, Schoen \cite{SchoenYau1994} proved that if 
$B_M(p,r)$ is stable, then
\[
\|A\|^2(p) \le \frac{C}{r^2},
\]
for some universal constant $C>0$. Related curvature estimates and rigidity phenomena were further developed by Schoen--Simon--Yau \cite{SchoenSimonYau1975} and Simons \cite{Simons}. More recently, comparison results for the first Dirichlet eigenvalue on minimal submanifolds under radial curvature bounds (see \cite[Corollary 1.12]{BessaMontenegro2003}) provided a converse-type criterion for stability in terms of upper bounds on $\|A\|^2$. Motivated by these developments and by the nonlinear spectral theory of the $p$-Laplacian, we now establish a $p$-analogue of this stability criterion 

\begin{definition}[$p$-stability]
Let $\Omega \subset M$ be a domain and let $\mathcal V \in L^\infty(\Omega)$ 
with $\mathcal V \ge 0$. 
We say that $\Omega$ is $p$-stable with respect to the potential $\mathcal V$ if for every $u \in \mathcal{C}_0^\infty(\Omega)$ we have
\[
Q_p(u) \coloneq \int_{\Omega} \Big( |\nabla u|^p - \mathcal V(x) |u|^p \Big) \ge 0
\]

\end{definition}
\begin{cor}[A $p$-stability criterion]\label{p-stability}
Assume the hypotheses of Theorem \ref{thm:pminsub_ball}.
Let $\varphi : M^m \hookrightarrow N^n$ be a minimal immersion and let 
$\Omega$ be a connected component of 
$\varphi^{-1}(B_N(p_0,r))$.
Fix $2 \le p < \infty$ and set $\mathcal V(x) = \|A(x)\|^p$ where $A$ denotes the second fundamental form of the immersion. If
\[
\sup_{x \in \Omega} \|A(x)\|^p 
\leq k^{p-2} \cdot \lambda_{1,p}^{\mathcal D}(B(r)).
\]
Then, $\Omega$ is $p$-stable with respect to the potential $\|A\|^p$.
\end{cor} 
\begin{remark}
When $p=2$, we recover precisely the classical stability criterion of \cite{BessaMontenegro2003}. Indeed, in this case the functional $Q_2(u)$ coincides with the quadratic form associated with the Jacobi operator $L=\triangle+\|A\|^2$ of a minimal hypersurface in Euclidean space. Thus, the condition $\sup_{\Omega}\|A\|^2 \leq \lambda_{1,2}^{\mathcal D}(B(r))$ is equivalent to the nonnegativity of the first Dirichlet eigenvalue of $L$ on $\Omega$, which is exactly the stability criterion proved in \cite[Corollary 1.12]{BessaMontenegro2003}.
\end{remark}

Moreover, in the Euclidean case, Theorem \ref{thm:pminsub_ball} yields a quantitative lower bound for the extrinsic radius of any
connected component $\Omega$ of $\varphi^{-1}\!\big(B_{\mathbb{R}^{n}}(p_{0},r)\big)$ (see Figure \eqref{confinement}). More precisely, there exists a radius $\widetilde r>0$ such that $\varphi(\Omega)\not\subset B_{\mathbb{R}^{n}}(p_{0},\widetilde r)$.

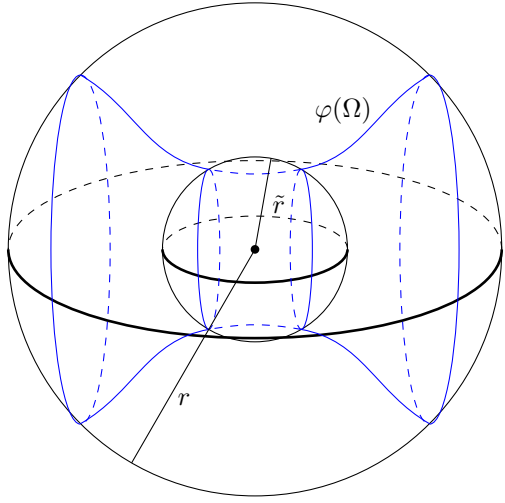
\begin{figure}[htb!]
    \centering
    \begin{tikzpicture}[scale=0.82,>=stealth]
		
		\path[fill] (0,0) circle (2 pt);
		\def\v{1.5}
		\path[draw,name path=c1] (0,0) circle (\v cm);
		\draw[dashed] (\v,0) arc (0:180:\v cm and 0.36*\v cm);
		\draw[line width=1pt] (\v,0) arc (360:180:\v cm and 0.36*\v cm);
		\def\k{4}
		\path[draw,name path=c2] (0,0) circle (\k cm);
		\draw[dashed] (\k,0) arc (0:180:\k cm and 0.36*\k cm);
		\draw[line width=1pt] (\k,0) arc (360:180:\k cm and 0.36*\k cm);
		
		\def\u{60}
		\path[draw,name path=arc1, blue] (\u:\v cm) arc (90:-90:0.12*\v cm and 0.865*\v cm);
		\path[draw,name path=arc2,dashed, blue] (\u:\v cm) arc (90:270:0.12*\v cm and 0.865*\v cm);
		\def\w{60}
		\path[draw,name path=arc1,dashed, blue] (180-\u:\v cm) arc (90:-90:0.12*\v cm and 0.865*\v cm);
		\path[draw,name path=arc2, blue] (180-\u:\v cm) arc (90:270:0.12*\v cm and 0.865*\v cm);
		
		\def\a{45}
		\path[draw,name path=arc1, blue] (\a:\k cm) arc (90:-90:0.12*\k cm and 0.708*\k cm);
		\path[draw,name path=arc2,dashed, blue] (\a:\k cm) arc (90:270:0.12*\k cm and 0.708*\k cm);
		
		\def\a{45}
		\path[draw,name path=arc1,dashed, blue] (180-\a:\k cm) arc (90:-90:0.12*\k cm and 0.708*\k cm);
		\path[draw,name path=arc2, blue] (180-\a:\k cm) arc (90:270:0.12*\k cm and 0.708*\k cm);
		
		\draw[blue] (\a:\k cm) to[out=210,in=10] (\u:\v cm);
		\draw[dashed,blue] (\u:\v cm) to[out=190,in=-10] (180-\u:\v cm);
		\draw[blue] (180-\u:\v cm) to[out=170,in=-30] (180-\a:\k cm);

		\draw[blue] (-\a:\k cm) to[out=150,in=-10] (-\u:\v cm);
		\draw[dashed,blue] (-\u:\v cm) to[out=170,in=10] (-180+\u:\v cm);
		\draw[blue] (-180+\u:\v cm) to[out=190,in=30] (-180+\a:\k cm);
		
		\draw (0,0) -- (-120:\k cm) node[pos=0.7,right] {$r$};
		\draw (0,0) -- (80:\v cm) node[pos=0.5,right] {$\Tilde{r}$};

		\node[right] (t1) at (70:0.60*\k) {$\varphi(\Omega)$};
		\end{tikzpicture}
\caption{Euclidian region of non-confinement for $\varphi(\Omega)$.}
    \label{confinement}
\end{figure}

\begin{cor}\label{radiuscor}
Assume the hypotheses of Theorem \ref{thm:pminsub_ball}
in the Euclidean case ($c=0$).
Let $\varphi : M^m \hookrightarrow \mathbb{R}^n$ be a minimal immersion and let 
$\Omega_{r}$ be a connected component of 
$\varphi^{-1}(B_{\mathbb{R}^n}(p_0,r))$.
Then
\[
r \geq \left( k^{p-2} \cdot \frac{\lambda_{1,p}^{\mathcal D}(B_1^m)}{\lambda_{1,p}(\Omega_{r})}\right)^{1/p},
\]
where $B_1^m$ denotes the unit ball in $\mathbb{R}^m$.
\end{cor} 

Another noteworthy result in the same direction of providing lower bounds for the fundamental tone is presented in \cite{BessaMontenegro2003}. The estimates obtained by the authors apply in a setting where the immersion $M \hookrightarrow N$ is locally bounded mean curvature. 
Subsequently, Bessa--Silvana \cite[Definition 1.5]{Bessa-Silvana} introduced the following notion of locally  bounded mean curvature $M \stackrel{\varphi}{\hookrightarrow} N\times\mathbb{R} $.
\begin{definition}
An isometric immersion $\varphi \colon M \hookrightarrow N \times \mathbb{R}$ is said to have locally bounded mean curvature if, for every $p \in N$ and every $r>0$, the quantity
\[
h(p,r) \coloneqq 
\sup \left\{ |H(x)| \; ; \; x \in \varphi(M)\cap \bigl( B_N(p,r)\times \mathbb{R} \bigr) \right\}
\]
is finite. Here, $B_N(p,r)$ denotes the geodesic ball in $W$ of radius $r$ centered at $p$.
\end{definition} 
Under these assumptions, the authors of \cite[Theorem 1.6]{Bessa-Silvana} establish an estimate for the bottom of the spectrum by means of \cite[Lemma 2.3]{BessaMontenegro2003} for a complete immersed $m$-submanifold of $N \times \mathbb{R}$ with
locally bounded mean curvature. In the present work, we prove an analogous estimate for the $p$-fundamental tone by using the result of \cite[Theorem 1.1]{LimaMontenegroSantos2010}.
\begin{theorem}\label{thmlowerptone}
Let $\varphi\colon M \hookrightarrow N \times \mathbb{R}$ be a complete  $m$-submanifold immersed with locally bounded mean curvature into $ N \times \mathbb{R}$. Assume that $N$ has a radial sectional curvature bounded above by $K_N(\partial_t,v) \leq c$ along the geodesics issuing from a point $x_0\in N$. Let $\Omega(r)\subset \varphi^{-1}\left(B_N(x_0,r)\times\mathbb{R}\right)$ be a connected component and suppose that $r \leq \min \left\{\operatorname{inj}_N(x_0), \pi/2\sqrt{c} \right\}$. Additionally, suppose that $r$ is chosen
\begin{enumerate}
\item If $|h(x_0,r)|< \Lambda < \infty$ let $r \leq (S_c / S'_c)(\Lambda/(m-2))$;
\item If $\displaystyle \lim_{r \to \infty}h(x_0,r) = \infty$ let $r \leq (S_c / S'_c)(h(x_0,r_{0})/(m-2))$, where $r_0$ is so that $(m-2)(S'_c/S_c)-h(x_0,r_0)=0.$
\end{enumerate}
Then, for every $p>1$, whenever either $1$ or $2$ holds, it follows that
\begin{equation}
 \lambda_{1,p}(\Omega(r)) \geq    \frac{1}{p^{p}}
\left((m-2)\frac{S'_c}{S_c}(r)-h(x_0,r)\right)^{p} > 0.
\end{equation}
\end{theorem} 

As a consequence of Theorem \ref{thmlowerptone}, which provides a lower bound for the $p$-fundamental tone, we obtain the following $p$-stability criterion for immersions with locally bounded mean curvature
\begin{cor}\label{corpstaby}
Assume the hypotheses of Theorem \ref{thmlowerptone}. Let $\varphi \colon M \hookrightarrow \mathbb{R}^{n}$ be a minimal immersion into Euclidean space. Suppose that the second fundamental form $A$ in $\Omega_r$ satisfies
\begin{equation*}
    \|A\| \leq \frac{m-1}{pr}.
\end{equation*}
Then $\Omega_r$ is $p$-stable with respect to the potential $\|A\|^{p}$.
\end{cor} 

Finally, by means of Theorem \ref{thm:pbarta}, we derive a nonlinear Kazdan--Kramer \cite{KazdanKramer} type characterization of the first Dirichlet eigenvalue of the $p$-Laplacian. In this framework, the existence of a blow-up solution constrains the source term to belong to a sharp spectral interval determined by $\lambda_{1,p}^{\mathcal D}(M)$, with spectral rigidity in the extremal regimes. Observe that, when $p=2$, we recover exactly the result of \cite[Theorem 1.13]{BessaMontenegro}.

\begin{theorem}\label{Thm:pKazdame}
Let $M$ be a bounded Riemannian manifold with smooth boundary, and let 
$\Psi \in \mathcal{C}^0(\overline{M})$. Consider the problem
\begin{equation}\label{intro:semilineareq}
\begin{cases}
\triangle_p u - (p-1)|\nabla u|^p = \Psi & \text{in } M;\\
u = +\infty & \text{on } \partial M.
\end{cases}
\end{equation}
If \eqref{intro:semilineareq} admits a smooth solution, then $\inf_M \Psi \leq \lambda_{1,p}^{\mathcal{D}}(M) \leq \sup_M \Psi$. Moreover, if either $\inf_M \Psi = \lambda_{1,p}^{\mathcal{D}}(M)$ or $\sup_M \Psi = \lambda_{1,p}^{\mathcal{D}}(M)$, then $\Psi \equiv \lambda_{1,p}^{\mathcal{D}}(M)$. Finally, if $\Psi \equiv \lambda$ is constant, then problem \eqref{intro:semilineareq} admits a solution if and only if $\lambda = \lambda_{1,p}^{\mathcal{D}}(M)$.
\end{theorem}

\section{Bartas’s Theorem for \texorpdfstring{\(p\)--}aLaplacian}

Picone's classical inequality was initially developed in \cite[p. 18]{Picone1909} to provide a new proof of Sturm's Theorem. The tool has become quite useful, being used in various contexts (see \cite{AllegrettoHuang1998, AroraGiacomoniWarnault2020, Jaros2015, lindqvist, Tadie2009}). The inequality in its classical form establishes that: if $u: \Omega \to \mathbb{R}$ and $v: \Omega \to \mathbb{R}$ are two differentiable functions such that $u \geq 0$ and $v > 0$, then in each connected component of $\Omega$ we have
\begin{equation}\label{piconeclassico}
| \nabla u |^2 + \dfrac{u^2}{v^2} | \nabla v |^2 - 2 \dfrac{u}{v} \left\langle \nabla u, \nabla v \right\rangle = | \nabla u |^2 - \left\langle \nabla v, \nabla \left( \dfrac{u^2}{v} \right) \right\rangle \geq 0.
\end{equation}
Moreover, the equality in \eqref{piconeclassico} holds if, and only if, $u$ is proportional to $v$. Note that the lower bound of the Barta inequality given in \eqref{eq:barta} can be obtained through \eqref{piconeclassico} by taking $u$ as the first eigenfunction of the Laplacian and integrating by parts the right-hand side of the equation \eqref{piconeclassico}. Subsequently, the inequality \eqref{piconeclassico} was extended in the following work \cite[Theorem 1.1]{AllegrettoHuang1998} ensuring that if $u \geq 0$ and $v > 0$ are differentiable functions, then
\begin{equation}\label{p-picone}
| \nabla u |^p + (p-1)\dfrac{u^p}{v^p} | \nabla v |^p - p \dfrac{u^{p-1}}{v^{p-1}} | \nabla v |^{p-2} \left\langle \nabla u, \nabla v \right\rangle = | \nabla u |^p - | \nabla v |^{p-2} \left\langle \nabla v, \nabla \left( \dfrac{u^p}{v^{p-1}} \right) \right\rangle \geq 0.
\end{equation}
Furthermore, the equality holds if and only if $u$ is proportional to $v$ in each connected component of the domain $\Omega$. When $p=2$ in \eqref{p-picone} we recover the classical Picone inequality \eqref{piconeclassico}. In Section 2 of \cite{AllegrettoHuang1998} we observed that from the inequality given in \eqref{p-picone} we obtain the following results for the $p$-Laplacian: simplicity of the principal eigenvalue, Barta's inequalities, Hardy's inequalities, and nonexistence of solutions. Based on \cite[Lemma 2.1]{AllegrettoHuang1998} the following result is a natural extension of \cite[Proposition 2.1]{BessaMontenegro} and \cite[Corollary 2.1]{AllegrettoHuang1998}, using, respectively, the same approximation technique in \cite{BessaMontenegro} and Picone's inequality for the $p$-Laplacian \eqref{p-picone} to extend Barta's inequality for the $p$-fundamental tone in bounded Riemannian domains.

\medskip
\noindent \textbf{Theorem 1.2.} \textit{Let $(M,g)$ be a smooth Riemannian manifold, and let $\Omega \subset M$ be a domain. Assume that $\eta \in \mathcal{C}^{1+\alpha}(\Omega)\cap\mathcal{C}^{0}(\overline{\Omega})$ with $\triangle_p \eta \in\mathcal{C}^0(\Omega)$
and $\eta>0$ in $\Omega$.
Then the $p$-fundamental tone $\lambda_{1,p}(\Omega)$ satisfies
\begin{equation}\label{eq:pBarta}
\lambda_{1,p}(\Omega)
\ge
\inf_{\Omega}
\left\{
-\frac{\triangle_p \eta}{\eta^{p-1}}
\right\}.
\end{equation}
Moreover, if $\Omega$ is bounded, the equality in \eqref{eq:pBarta}
holds if and only if $\eta=\beta u$ for some constant $\beta>0$,
where $u \in W^{1,p}_0(\Omega)$ is a nonnegative minimizer of the Rayleigh quotient.
}
\begin{proof} Let $\Omega \subset M$ be a domain. Let $\varphi \in \mathcal{C}_0^{\infty}(\Omega)$ with $\varphi \ge 0$ and choose $\eta \in \mathcal{C}^{1+\alpha}(\Omega)$ with $\triangle_p \eta \in \mathcal{C}^{0}(\Omega)$ such that $\eta > 0$ in $\Omega$. Since $\operatorname{supp}(\varphi) \Subset \Omega$ and $\eta$ is continuous and strictly positive in $\Omega$, there exists $c>0$ such that $\eta \ge c$ on $\operatorname{supp}(\varphi)$. Hence, the function $\varphi^p/\eta^{p-1}$ is well defined and belongs to $W^{1,p}(\Omega)$. Moreover, since its support is contained compactly in $\Omega$, we conclude that $\varphi^p/\eta^{p-1} \in W_0^{1,p}(\Omega)$. Therefore, it may be used as a test function in the weak formulation of the $p$-Laplacian. Applying Picone’s identity \eqref{p-picone} and integrating over $\Omega$, we obtain
\begin{equation}\label{huang&allegeto1}
\begin{aligned}
   0 &\leq  \int_{\Omega} \left( | \nabla \varphi |^p + (p-1)\dfrac{ \varphi^p}{\eta^p} | \nabla \eta |^p - p \dfrac{ \varphi^{p-1}}{\eta^{p-1}} | \nabla \eta |^{p-2} \left\langle \nabla  \varphi, \nabla \eta \right\rangle \right)  d\mu \\
   &=\int_{\Omega} | \nabla \varphi |^p - | \nabla \eta |^{p-2} \left\langle \nabla \eta, \nabla \left( \dfrac{ \varphi^p}{\eta^{p-1}} \right) \right\rangle d\mu \\
   &= \int_{\Omega} | \nabla \varphi |^p d\mu - \int_{\Omega} \left\langle | \nabla \eta |^{p-2} \nabla \eta, \nabla \left( \dfrac{ \varphi^p}{\eta^{p-1}} \right) \right\rangle d\mu \\
    &= \int_{\Omega} | \nabla \varphi |^p d\mu + \int_{\Omega}(\triangle_p \eta ) \cdot \dfrac{ \varphi^p}{\eta^{p-1}} \ d\mu  \\
    &\leq \int_{\Omega} | \nabla \varphi |^p d\mu - \inf_{\Omega} \left\{-\dfrac{\triangle_p \eta}{\eta^{p-1}} \right\} \cdot \int_{\Omega}| \varphi|^pd\mu. 
    \end{aligned}
\end{equation}
Where $d\mu$ denotes the Riemannian volume measure on $\Omega$. Consequently, from \eqref{huang&allegeto1}, for every $\varphi \in \mathcal{C}_0^{\infty}(\Omega)$, we obtain
\begin{equation}\label{huang&allegeto}
   \dfrac{\displaystyle\int_{\Omega} | \nabla \varphi |^p d\mu}{\displaystyle\int_{\Omega}| \varphi|^pd\mu}\geq \inf_{\Omega} \left\{ -\frac{\triangle_p \eta}{\eta^{p-1}}  \right\}.
\end{equation}
Taking the infimum over $\varphi \in \mathcal{C}_0^\infty(\Omega)\setminus\{0\}$
in \eqref{huang&allegeto} and using the density of $\mathcal{C}_0^\infty(\Omega)$ in $W_0^{1,p}(\Omega)$ (see \cite[Corollary A.4]{le}), we conclude that \eqref{eq:pBarta} holds.
Assume that $\Omega$ is bounded and that $\eta=\beta u$ for some $\beta>0$,
where $u\in W_0^{1,p}(\Omega)$ is a nonnegative minimizer of the Rayleigh quotient, i.e.,
\[
\lambda_{1,p}(\Omega) =\frac{\displaystyle \int_\Omega \ |\nabla u|^pd\mu}{\displaystyle \int_\Omega |u|^p d\mu}.
\]
Since $u$ satisfies $\triangle_p u = -\lambda_{1,p}(\Omega)\,u^{p-1}$ in $\Omega$, the $(p-1)$-homogeneity of the $p$-Laplacian, $\triangle_p(\beta u)=\beta^{p-1}\triangle_p u$, yields the following identity
\[
-\triangle_p \eta
= -\triangle_p(\beta u)
= -\beta^{p-1}\triangle_p u
= \beta^{p-1}\lambda_{1,p}(\Omega)\,u^{p-1}
=\lambda_{1,p}(\Omega)\eta^{\,p-1}
\quad\text{in }\Omega
\]
in the weak sense. In particular, wherever $u>0$ (equivalently, $\eta>0$),
\[
\lambda_{1,p}(\Omega) = -\frac{\triangle_p\eta}{\eta^{p-1}}.
\]
Hence, $\lambda_{1,p}(\Omega) = \inf_\Omega\left\{-\triangle_p\eta/\eta^{p-1}\right\}$, and equality holds in \eqref{eq:pBarta}. It remains to prove the converse implication. Assume that equality holds in \eqref{eq:pBarta}. Since $\Omega$ is bounded, the infimum in the Rayleigh quotient is achieved. By standard variational arguments, $u$ is a weak solution of $\triangle_p u = -\lambda_{1,p}(\Omega)\,u^{p-1}$ in $\Omega$. Let $\{\varphi_n\}_{n\in\mathbb{N}}\subset \mathcal{C}_0^\infty(\Omega)$ be a sequence with $\varphi_n\ge0$ and
$\varphi_n\to u$ strongly in $W_0^{1,p}(\Omega)$. For each $n$, Picone's identity yields
\[
0\le \int_\Omega \Bigl(|\nabla \varphi_n|^p
-|\nabla \eta|^{p-2}\bigl\langle \nabla\eta,\nabla(\varphi_n^p/\eta^{p-1})\bigr\rangle\Bigr)d\mu.
\]
Proceeding as in \eqref{huang&allegeto1}, and using the weak formulation of $\triangle_p\eta$ with the test
function $\varphi_n^p/\eta^{p-1}\in W_0^{1,p}(\Omega)$, it follows that
\[
0 \leq \int_\Omega |\nabla \varphi_n|^p d\mu
-\int_\Omega \left(-\frac{\triangle_p\eta}{\eta^{p-1}}\right)|\varphi_n|^p d\mu.
\]
Passing to the limit as $n\to\infty$ and using $\varphi_n\to u$ in $W_0^{1,p}(\Omega)$ gives the following 
\begin{equation}\label{inebar1}
0 \leq \int_\Omega |\nabla u|^p d\mu
-\int_\Omega \left(-\frac{\triangle_p\eta}{\eta^{p-1}}\right) |u|^p d\mu.    
\end{equation}
Let $\lambda \coloneq \inf_{\Omega}\left\{-\triangle_p\eta/\eta^{p-1}\right\}$. Since equality holds in \eqref{eq:pBarta}, we have $\lambda_{1,p}(\Omega)=\lambda$.
Moreover,
\begin{equation}\label{inebar2}
\int_\Omega |\nabla u|^p d\mu=\lambda_{1,p}(\Omega)\int_\Omega |u|^p d\mu
=\lambda \int_\Omega |u|^p d\mu \leq \int_\Omega\left(-\dfrac{\triangle_p\eta}{\eta^{p-1}}\right)|u|^p d\mu.    
\end{equation}
Combining this with the previous inequality \eqref{inebar1} and \eqref{inebar2} we conclude that equality must hold in both inequalities. Hence, the nonnegative integrand in Picone's identity must vanish a.e. in $\Omega$, and therefore
\begin{equation}\label{equalitypicone}
| \nabla u |^p + (p-1)\frac{u^p}{\eta^p}| \nabla \eta |^p
- p \frac{u^{p-1}}{\eta^{p-1}}
\left\langle \nabla u, |\nabla \eta |^{p-2}\nabla \eta \right\rangle = 0.
\end{equation}
Henceforth, all equalities and inequalities are understood to hold almost everywhere in $\Omega$. Thus, we can rewrite the equation \eqref{equalitypicone} as follows
\begin{equation*}
\underbrace{| \nabla u |^p - p \frac{u^{p-1}}{\eta^{p-1}}
| \nabla u | | \nabla \eta |^{
p-2} | \nabla \eta |
     + (p-1)\frac{u^p}{\eta^p}| \nabla \eta |^p}_{\mathcal{A}} + \underbrace{p \frac{u^{p-1}}{\eta^{p-1}}| \nabla \eta |^{p-2}
\left( | \nabla u |  | \nabla \eta | - \left\langle \nabla u , \nabla \eta \right\rangle \right)}_{\mathcal{B}} =0.
\end{equation*}

Using Young's inequality ($\displaystyle ab \leq a^p/p + b^q/q$ where $q$ is the Hölder conjugate of $p$) with $a \coloneq | \nabla u |$ and $b \coloneq \left( u | \nabla \eta |/\eta \right)^{p-1}$, we guarantee that $\mathcal{A}(x) \geq 0$. Note also that $\mathcal{B}(x)\geq 0$, therefore, $\mathcal{A}(x) = \mathcal{B}(x)=0$ a.e. in $\Omega$. From $\mathcal{B}=0$ we obtain the following
\[
|\nabla u|\cdot|\nabla\eta|=\langle \nabla u,\nabla\eta\rangle,
\]
hence $\nabla u$ and $\nabla\eta$ are parallel a.e.\ in $\Omega$. In particular, there exists a measurable function $\alpha:\Omega\to[0,\infty)$ such that
\begin{equation}\label{parallel}
\nabla \eta = \alpha \cdot \nabla u.
\end{equation} 
Set $\Sigma \coloneq \{x\in\Omega : |\nabla u(x)|\,|\nabla\eta(x)|>0\}$. 
On $\Sigma$ we may divide by $|\nabla\eta|$, and since $\mathcal{A}=0$ a.e.\ in $\Omega$ we can divide $\mathcal{A}$ by $(u|\nabla\eta|/\eta)^p$ to obtain, for a.e.\ $x\in\Sigma$,
\begin{equation}\label{quociente}
\left(\frac{\eta}{u}\frac{|\nabla u|}{|\nabla\eta|}\right)^p
- p\frac{\eta}{u}\frac{|\nabla u|}{|\nabla\eta|}
+ (p-1)=0 .
\end{equation}
Let  $y \coloneq \left( \eta | \nabla u | \right)/\left( u | \nabla \eta | \right)$. Then \eqref{quociente} becomes
\begin{equation}\label{eqdio}
y^p - py + (p-1)=0. 
\end{equation}
Since $t=1$ is a positive root of $t^p-pt+(p-1)=0$ and the function
$t\mapsto t^p$ is strictly convex on $(0,\infty)$, this equation admits a unique positive root. Hence $y=1$ a.e.\ in $\Sigma$, and therefore
\begin{equation}\label{a}
|\nabla u|=\frac{u}{\eta}\,|\nabla\eta|.
\end{equation}
On $\Omega\setminus\Sigma$ we have $|\nabla u|\,|\nabla\eta|=0$. If there exists $x$ such that $|\nabla\eta(x)|=0$, then \eqref{equalitypicone} reduces to
$|\nabla u(x)|^p=0$; similarly, if $|\nabla u(x)|=0$, then it reduces to
$(p-1)(u^p/\eta^p)|\nabla\eta(x)|^p=0$. Hence $|\nabla u|=|\nabla\eta|=0$ a.e. on $\Omega\setminus\Sigma$. Furthermore, note that if  $x \in \Sigma$, we have $\alpha(x)>0$ and combining equations \eqref{parallel} and \eqref{a}, we obtain $\alpha= \eta/u$. Since $u,\eta>0$ in $\Omega$ and $u,\eta\in \mathcal{C}^{1,\alpha}_{\mathrm{loc}}(\Omega)$,
we have $\log u,\log\eta \in \mathcal{C}^{1,\alpha}_{\mathrm{loc}}(\Omega)$ and
\[
\nabla(\log\eta-\log u)
= \frac{\nabla\eta}{\eta}-\frac{\nabla u}{u}=0.
\]
Therefore, $\log\eta-\log u$ is constant on each connected component of $\Omega$, and since $\eta,u>0$ are continuous,
we obtain $\eta=\beta\,u$ in $\Omega$ for some constant $\beta>0$.
\end{proof}

Next, we briefly review the theory that will be used in Section \ref{Sec3}. For a comprehensive account of the theory discussed in the following, we refer the interested reader to \cite{BessaMontenegro, LimaMontenegroSantos2010, lindqvist, SchoenYau1994}.

\begin{definition}[Weak divergence]
\label{def:weakdiv}
Let $(M,g)$ be a Riemannian manifold and let $X \in L_{\mathrm{loc}}^{1}(M)$ be a vector field, in the sense that $\|X\| \in L_{\mathrm{loc}}^{1}(M)$. 
We say that a function $h \in L_{\mathrm{loc}}^{1}(M)$ is the weak divergence of $X$, and we write $h = \operatorname{Div} X$, if for every test function $\eta \in \mathcal{C}_0^{\infty}(M)$ the following identity holds:
\[ \int_{M} \eta \operatorname{Div}(X) dM = -\int_{M} \langle \nabla \eta, X \rangle dM. \]
\end{definition}
We denote by $\mathcal{W}^{1,1}(M)$ the space of vector fields on $M$
that admits a weak divergence. 
\begin{remark}\label{caracteeizationDiv}
If $X \in \mathcal{W}^{1,1}(M)$ and $f \in \mathcal{C}^{\infty}(M)$ then $fX \in \mathcal{W}^{1,1}(M)$ with $\operatorname{Div}(fX)= \langle \nabla f , X \rangle +f\operatorname{Div}(X)$. Futhermore, if $f \in \mathcal{C}_0^{\infty}(M)$, then
\begin{equation}
    \int_M \operatorname{Div}(fX) dM =\int_M \Big[ \langle \nabla f , X \rangle +f\operatorname{Div}(X)\Big] dM = 0.
\end{equation}
Conversely, if $fX \in \mathcal{W}^{1,1}(M)$ for all $f \in \mathcal{C}_0^{\infty}(M)$,
then $X \in \mathcal{W}^{1,1}(M)$. In this case, the weak divergence of $X$ is
given by
\begin{equation*}
\operatorname{Div}(X) = \sum_{i=1}^{\infty} \operatorname{Div}\!\left(\xi_i X\right),  
\end{equation*} where $\{\xi_i\}_{i\in\mathbb{N}}$ is a partition of unity subordinated to a locally finite open covering of $M$. 
\end{remark}
A sufficient condition ensuring that a vector field $X$ belongs to the Sobolev space $\mathcal{W}^{1,1}(M)$ is established in \cite[Lemma 3.5]{BessaMontenegro}, leading to the following criterion

\begin{definition}
Let $(M, g)$ be a Riemannian manifold. A vector field $X$ on a domain $\Omega \subset M$ is said to be of class $L^{\infty}(\Omega)$ if
\[
\|X\|_\infty \coloneqq \operatorname*{ess\,sup}_{x\in \Omega} |X(x)| < \infty.
\]
Where $|X(x)| = g(X(x),X(x))^{1/2}$.
\end{definition} 
\begin{prop}[Bessa--Montenegro \cite{BessaMontenegro}]\label{LemmaSobolev}
Let $\Omega \subset M$ be a bounded domain in a smooth Riemannian manifold $M$,
and let $\mathcal{F} \subset M$ be a closed subset such that
$\mathcal{H}^{n-1}(\mathcal{F} \cap \Omega) = 0$.
Let $X$ be a vector field satisfying
$X \in \mathcal{C}^1(\Omega \setminus \mathcal{F}) \cap L^{\infty}(\Omega)$
and $\operatorname{div}(X) \in L^{1}(\Omega)$.
Then $X \in \mathcal{W}^{1,1}(\Omega)$, and its weak divergence satisfies
$\operatorname{Div}(X) = \operatorname{div}(X)$ in $\Omega \setminus \mathcal{F}$.
\end{prop}
In addition to extension of Barta's Theorem, the work \cite{BessaMontenegro} provides a version of this theorem for smooth fields, giving an estimate for the fundamental tone through a min-max technique in fields $X$ defined in $\mathcal{W}^{1,1}(M)$. Furthermore, the result \cite[Theorem 1.7]{BessaMontenegro} remains valid for the $p$-fundamental tone and was proven in \cite[Theorem 1.3]{LimaMontenegroSantos2010}. For the reader’s convenience, we record the result below
\begin{theorem}[\cite{LimaMontenegroSantos2010}]\label{Thm:BMS}
Let $(M, g)$ be a Riemannian manifold. Then the following estimate for the $p$-fundamental tone of $M$ holds
\[
\lambda_{1,p}(M) \geq \sup \left\{\inf_{\Omega} \bigl( (1-p)\|X\|^q+\operatorname{Div}(X) \bigr), X \in \mathcal{W}^{1,1}(M)\right\}.
\]
Where $1 / p+1 / q=1$.    
\end{theorem}

The estimate above is sharp because, if $u$ is the first Dirichlet eigenfunction of the \(p\)-Laplacian. That is,
\[ \triangle_p u = \lambda_{1,p}^{\mathcal{D}}(\Omega) |u|^{p-2}u.  \]
Consider $X$ the smooth vector field given by 
\[ X = -\frac{|\nabla u|^{p-2} \nabla u}{|u|^{p-2} u}.  \]
Since $u \in W_0^{1,p}(\Omega)$ is a solution of \eqref{Dirichleteigenvalue}, it follows by \cite[Lemma 5]{le} that $u \in \mathcal{C}^{1,\alpha}(\Omega)$, therefore $X \in \mathcal{W}^{1,1}(\Omega)$. Thus, using \cite[Remark 2.1]{LimaMontenegroSantos2010} we obtain that $\lambda_{1,p}^{\mathcal{D}}(\Omega) = (1-p)\|X\|^q+\operatorname{div}(X)$. Furthermore, we can repeat the same argument from the proof of \cite[Theorem 1.3]{LimaMontenegroSantos2010} replacing $M$ with $M \setminus \mathcal{F}$ where $\mathcal{F}$ is a set of zero Riemann measure, then the result given in Theorem \ref{Thm:BMS} remains valid outside of that set.
\begin{remark}\label{rmk:BMS}
Let $(M, g)$ be a Riemannian manifold. Then the following estimate for the $p$-fundamental tone of $M$ holds
$$
    \lambda_{1,p}(M) \geq \sup \left\{\inf_{M \setminus \mathcal{F}} \bigl( (1-p)\|X\|^q+\operatorname{Div}(X) \bigr), X \in \mathcal{W}^{1,1}(M)\right\}.
$$
Where $1 / p+1 / q=1$ and $\mathcal{F}$ has zero Riemannian volume. 
\end{remark}

Barta’s formulation via vector fields, as presented originally in \cite{cheung} in submanifolds of the  Hyperbolic space and generalized in \cite{BessaMontenegro}
and \cite{LimaMontenegroSantos2010}, allows one to estimate the $p$-fundamental tone
without requiring positive test functions or any boundary regularity.
Moreover, it naturally incorporates geometric information through the divergence of $X$, making it possible to derive explicit spectral estimates in terms of curvature,
immersions, or volume comparison. The inequality for the $p$-fundamental tone established in the work of Lima--Santos--Montenegro \cite[Theorem 1.1]{LimaMontenegroSantos2010} is the key ingredient in the proof of Theorem \ref{thmlowerptone}. This estimate is obtained through a max--min technique applied to the class of all admissible smooth vector fields. More precisely, it yields the following estimate

\begin{theorem}[Lima--Santos--Montenegro]\label{LimaSantosMontenegro}
Let $\Omega \subset M$ be a domain $(\partial \Omega \neq \emptyset)$ in a smooth Riemannian manifold $(M,g)$. Denote by $\mathfrak{X}(\Omega)$ the set of all smooth vector fields $X \in L^{\infty}(\Omega)$ .
Then
\[
 \lambda_{1,p}(\Omega) \geq \dfrac{1}{p^p} \left( \sup_{X \in \mathfrak{X}(\Omega) } \left\{ \frac{   \displaystyle \inf_{\Omega}{\operatorname{div}(X)}}{  \Vert X \Vert_{\infty}} \right\}  \right)^p. 
\]
\end{theorem}

\section{Proof of the Main Results}\label{Sec3} For the proof of Theorem \ref{thm:Cheng}, we employ the same techniques used in \cite{BessaMontenegro}, in conjunction with Theorem \ref{Thm:BMS} and Remark \ref{rmk:BMS}. The rigidity statement follows analogously to the argument in \cite[Theorem 1.8]{BessaMontenegro}. To guarantee the isometry required in Theorem \ref{thm:Cheng}, we use a result of Bishop (see \cite[Theorem 3.1]{BessaMontenegro} or \cite{BishopCrittenden1964}) together with the Laplacian Comparison Theorem \cite{GreeneWu1979}. For the proof of Theorem \ref{thm:pminsub_ball}, we make use of the Hessian Comparison Theorem, as presented in \cite{SchoenYau1994}, to obtain the desired comparison estimates, and we again apply Bishop’s theorem to ensure rigidity. Moreover, Theorem \ref{thmlowerptone} is established by means of the inequality given in \cite[Theorem 1.1]{LimaMontenegroSantos2010}, within the same admissible class considered in the proof of \cite[Theorem 1.6]{Bessa-Silvana}. Finally, Theorem \ref{Thm:pKazdame} establishes a correspondence between the first Dirichlet eigenvalue problem and a blow-up analysis of solutions. This correspondence is ensured by the Barta-type characterization provided by Theorem \ref{thm:pbarta}. We now state the main results that will be used throughout this section. Consider the function
\begin{equation}\label{espaceformfunc}
S_c(t) \coloneq
\begin{cases}
\dfrac{1}{\sqrt{c}}\sin(\sqrt{c}t), & c>0,\\
t, & c=0,\\
\dfrac{1}{\sqrt{-c}}\sinh(\sqrt{-c}t), & c<0.
\end{cases}    
\end{equation}

The next three lemmas are well known, and their proofs can be found, respectively, in  \cite{BishopCrittenden1964}, \cite{GreeneWu1979}, and \cite{SchoenYau1994}. Bishop’s theorem will be used to ensure the rigidity required in Theorems \eqref{thm:Cheng} and \eqref{thm:pminsub_ball}. The Laplacian Comparison Theorem and the Hessian Comparison Theorem will play a key role in establishing the comparison estimates for the $p$-fundamental tone in Theorems \eqref{thm:Cheng} and \eqref{thm:pminsub_ball}, respectively.

\begin{lem}[Bishop] \label{thm:bishop}
Let $M$ be an $m$-dimensional Riemannian manifold and let $\gamma_\xi$ be a unit-speed
geodesic issuing from a point $p\in M$ in the direction $\xi\in \mathbb{S}^{m-1}$.
Assume that the radial sectional curvatures along $\gamma_\xi$ satisfy
\[
\bigl\langle \mathrm{R}\bigl(\gamma_\xi'(t),v\bigr)\gamma_\xi'(t),v\bigr\rangle
\le c|v|^2,
\qquad \forall\, t\in(0,r)
\]
for every $v\perp \gamma_\xi'(t)$, and that $S_c(t)$ does not vanish on $(0,r)$.
Then
\begin{align}
\left(\frac{\sqrt{g(t,\xi)}}{S_c^{m-1}(t)}\right)' &\geq 0 
\qquad \text{on } (0,r), \label{eq:bishop1}\\
\sqrt{g(t,\xi)}-S_c^{m-1}(t) &\geq 0
\qquad \text{on } (0,r]. \label{eq:bishop2}
\end{align}
Moreover, equality in either \eqref{eq:bishop1} or \eqref{eq:bishop2} at some point
$t_0\in(0,r)$ holds if and only if
\[
\mathcal{R}=c \cdot \mathrm{Id}
\quad \text{and} \quad
\mathcal{A}=S_c \cdot \mathrm{Id}
\quad \text{on } [0,t_0].
\]
\end{lem}

\begin{lem}[Laplacian Comparison Theorem] Let $M$ be a complete Riemannian manifold and take $x_0$, $x_1 \in M$. Let $\gamma\colon\left[0, t(x_1)\right) \to \mathbb{R}$ be a minimizing geodesic joining $x_0$ and $x_1$ where $t(x)$ is the distance function $t(x):=\operatorname{dist}_{M}(x,x_0).$ Denote by  $K_M(\partial_t,x)$ the radial sectional curvature along the geodesic $\gamma$ and assume that $K_M(\partial_t,x) \leq K_f(\partial_t,t(x))$ then
\begin{equation*}
    \triangle t(x) \geq (n-1)\dfrac{f'(t(x))}{f(t(x))}. 
\end{equation*}
In particular, when $K_M$ is constant $f(r) = S_c(r)$ with $S_c$ given in \eqref{espaceformfunc}.
\end{lem}

\begin{lem}
[Hessian Comparison Theorem]\label{lem:hessian-comparison}
Let $(M,g)$ be a complete Riemannian manifold and choose $x_0,x_1\in M$. 
Denote by $t(x)\coloneqq \operatorname{dist}_M(x_0,x)$ the distance function from $x_0$, and let $\gamma\colon [0,\rho(x_1)]\to M$ be a minimizing geodesic joining $x_0$ to $x_1$. Let $K$ denote the sectional curvatures of $M$ along $\gamma$. Then the Hessian of $t$ at $\gamma(t(x))$ for all $X\perp \gamma'(t)$ satisfies the following inequality
\begin{equation*}
\operatorname{Hess} t(x)(X,X) \geq \frac{S'_c(t)}{S_c(t)} \cdot \Vert X \Vert .  \end{equation*} 
Where $S_c$ is a model function given in \eqref{espaceformfunc}.
  
\end{lem}

\subsection{Proof of Theorem \ref{thm:Cheng}}
\noindent \textbf{Theorem 1.4}
\textit{ Let $(M^m,g)$ be a Riemannian $m$-manifold and let $p_0\in M$. Assume that the radial sectional curvature of $M$ satisfies $K(\partial_t,v) \leq c$, $x\in B_M(p_0,r)\setminus \operatorname{Cut}(p_0)$, $v\perp \partial_t,\ |v|\leq1$, where $\partial_t$ denotes the radial unit vector field along minimal geodesics issuing from $p_0$. Let $\mathbb M^m(c)$ be the simply connected $m$-dimensional space form of constant sectional curvature $c$, and suppose that $\mathcal{H}^{m-1}\big(\operatorname{Cut}(p_0)\cap B_M(p_0,r)\big)=0$. Then the bottom of the Dirichlet spectrum of the geodesic ball $B_M(p_0,r)$ satisfies
\begin{equation}
\lambda_{1,p}\big(B_M(p_0,r)\big) \geq \lambda_{1,p}^{\mathcal{D}}\big(B(r)\big).    \end{equation}\label{comparison}
Moreover, equality holds if and only if the geodesic balls
$B_M(p_0,r)$ and $B(r)$ are isometric.}

\begin{proof}
Let $\mathbb{M}^m(c)$ be the $m$-dimensional model space of constant curvature $c$, with base point $x_0$ and radial coordinate given by $t(x):=\operatorname{dist}_{\mathbb{M}^n(c)}(x,x_0).$
Let $B(r)$ denote the geodesic ball of radius $r$ centered at $x_0$. Let us consider the eigenvalue problem associated with the model space $\mathbb{M}^m(c)$
\begin{equation}\label{eq:p-eig-modelo}
\begin{cases}
 \triangle_p u =  - \lambda_{1,p}^{\mathcal{D}}(B(r))\,|u |^{p-2}u 
& \text{em } B(r);\\
u = 0 & \text{em } \partial B(r).
\end{cases}
\end{equation}

We know from \cite[Lemma 5]{lindqvist92} that the problem above admits at least one solution $u$, such that $u$ does not change by sign in $B(r)$. Therefore, without loss of generality, we can assume $u\geq 0$ in $B(r) \cup \partial B(r)$, and also that $u>0$ inside $B(r)$. Furthermore, it follows from \cite[Theorem 1]{Bhattacharya1988} that there exists a solution to the eigenvalue problem \eqref{eq:p-eig-modelo} that is radial in $B(r)$. Since the first eigenvalue of the Dirichlet \(p\)-Laplacian is simple \cite[Theorem 1.3]{lindqvist}, we have that the solution $u$ must be radial. Thus, there exists $\omega:[0,r] \to \mathbb{R}$, such that $u(x) = \omega (t(x)).$ Furthermore, $\omega>0$ in $(0,r)$, $\omega(r)=0$. Let $u(x)=\omega(t(x))>0$ be  defined in $\mathbb{M}^{m}(c)$. Let $x\in B(r)\subset \mathbb{M}^{m}(c)$ and  $$\{e_1=\partial/\partial t, e_2=\partial/\partial\theta_1,\ldots, e_{n}=\partial/\partial_{m-1}\},$$ be an orthonormal basis of $T_x \mathbb{M}^{m}(c)$.
\begin{equation*}
\begin{aligned}
    \triangle_p(u)&= \diver (| \nabla u |^{p-2}\nabla u ) \\
    &= \sum_{i=1}^{m}
    \langle \nabla_{e_i}(|\nabla u|^{p-2}\nabla u), e_i\rangle  \\
    &=  \sum_{i=1}^{m} e_i(| \nabla u |^{p-2})\langle \nabla u, e_i\rangle + \sum_{i=1}^{m}|\nabla u |^{p-2}\langle \nabla_{e_i}( \nabla u),\, e_i\rangle \\
    &=  e_1(| \nabla u |^{p-2})\langle \nabla u, e_1\rangle + \sum_{i=2}^{m} \stackrel{=0}{e_i(| \nabla u |^{p-2})}\langle \nabla u, e_i\rangle + | \nabla u |^{p-2}\langle \nabla_{e_1}( \nabla u), e_1\rangle \\ 
    & \; \; \; \; +  \sum_{i=2}^{m}| \nabla u |^{p-2}\langle \nabla_{e_i}( \nabla u),\, e_i\rangle \\
&= (p-2)|\omega'(t)|^{p-2} \omega''(t)+|\omega'(t)|^{p-2}\left(\omega''(t)+\omega'(t)\stackrel{=0}{{\rm Hess} \,t(e_1,e_1)}\right) \\ 
& \; \; \; \; + |\omega'(t)|^{p-2}\omega'(t)\sum_{i=2}^{m}{\rm Hess} \, t(e_i,e_i) \\
    &= (p-2)|\omega'(t)|^{p-2}\omega''(t)+|\omega'(t)|^{p-2} \left[  \omega''(t)+ (m-1)\omega'(t) \frac{S_{c}'(t)}{S_{c}(t)}\right].
    \end{aligned}
\end{equation*}
Therefore, the eigenvalue equation $\triangle_p u +\lambda_{1,p}^{\mathcal{D}}(B(r))| u |^{p-2}u=0$ reads as
\begin{equation}\label{eq:eigenp}
    (p-2) |\omega'(t)|^{p-2}\omega''(t)+|\omega'(t)|^{p-2}\left[  \omega''(t) + (m-1) \omega'(t) \frac{S_{c}'(t)}{S_{c}(t)}\right]+ \lambda_{1,p}^{\mathcal{D}}\big(B(r)\big)\omega^{p-1}(t)=0.
\end{equation}
A direct computation yields $\left(|\omega'|^{p-2}\omega'\right)' = (p-1)|\omega'|^{p-2}\omega''$, and therefore \eqref{eq:eigenp} can be rewritten as
\begin{equation}\label{eqspace2}
\left( |\omega'(t)|^{p-2} \omega'(t) \right)' + (m-1) \frac{S_{c}'(t)}{S_{c}(t)} |\omega'(t)|^{p-2} \omega'(t)   + \lambda_{1,p}^{\mathcal{D}}(B(r))(\omega(t))^{p-1}=0.
\end{equation}
Since $\omega$ is radial and, by \cite[Lemma 5]{le}, $\omega \in \mathcal{C}^{1,\alpha}(B(r))$, regularity at the center implies that $\omega'(0)=0$. Using equation \eqref{eqspace2} in the model space, we observe that
\begin{equation*}
\left( (S_{c}(t))^{m-1}|\omega'(t)|^{p-2} \omega'(t) \right)'  +\lambda_{1,p}^{\mathcal{D}}(B(r))(S_{c}(t))^{m-1}(\omega(t))^{p-1}=0.
\end{equation*}
Therefore, we can integrate from $0$ to $t$, with $t \in (0,r)$, and using that $\omega'(0)=0$, we obtain
\begin{equation*}
    \omega'(t) \leq -\frac{\lambda_{1,p}^{\mathcal{D}}(B(r))(S_{c}(t))^{m-1}}{|\omega'(t)|^{p-2}}\int_0^{t} s^{m-1}(\omega(s))^{p-1}ds,
\end{equation*}
ensuring that $\omega'(t) < 0$ for $t \in (0,r)$ and by Hopf's Lemma (see \cite[Theorem 5]{Vazquez}) we have that $\omega'(t)=0$ iff $t=0$. On the other hand,
\[ B_M(r) \setminus \mathrm{Cut}(p_0) = \exp_{p_0} ( \{ t \xi \in T_{p_0}M; \, 0 \leq t < \min \{r,\text{d}\xi \}; \ | \xi |=1 \} ). \]
Now, consider the function $\overline{\omega} \colon B_M(p_0,r) \to [0,\infty)$ defined by
\begin{equation}\label{modelfunction}
\psi(x) \coloneq \left\{\begin{array}{cl}
\omega(t) & \text { if } x=\exp_{p_0}(t \xi), \ t \in[0, d(\xi)) \cap[0, r]; \\
0 & \text { if } x=\exp_{p_0}(d(\xi)\xi) \in \operatorname{Cut}(p_0).
\end{array}\right.
\end{equation}
Define the vector field $X$ on $B_M(p_0,r)$ by
\begin{equation}\label{Field}
X(x) \coloneq
\begin{cases}
- \dfrac{|\nabla \psi|^{p-2}\nabla\psi}{|\psi|^{p-2}\psi}, 
& \text{if } x\in B_M(p_0,r)\setminus\big(\{p_0\}\cup \operatorname{Cut}(p_0)\big); \\
0,
& \text{if } x\in B_M(p_0,r)\cap\big(\{p_0\}\cup \operatorname{Cut}(p_0)\big).
\end{cases}
\end{equation}
The vector field $X \in \mathcal{W}^{1,1}(B_M(p_0,r))$, we shall establish this result later. Setting $\mathcal{F} \coloneq \{p_0\}\cup \operatorname{Cut}(p_0)$ as a consequence, we have $\operatorname{div}(X) \in L^{1}(\Omega)$. It then follows from Proposition \ref{LemmaSobolev} that $X \in \mathcal{W}^{1,1}(\Omega)$, and that the weak divergence $\operatorname{Div}(X)$ coincides with the classical divergence $\operatorname{div}(X)$ in $\Omega \setminus \mathcal{F}$. By hypothesis, we have $\mathcal{H}^{n-1}\bigl( B_M(p_0,r) \cap \mathcal{F} \bigr)=0$
using Theorem \ref{Thm:BMS} together with Remark \ref{rmk:BMS}, we obtain the following estimate
\begin{eqnarray}\label{Newton-Montenegro-Barnabé}
\begin{aligned}
    \lambda_{1,p}(B_M(p_0,r)) &\geq \inf_{\Omega} \{ (1-p)\|X\|^q+\operatorname{Div}(X) \} \\
    &= \inf_{\Omega \setminus \mathcal{F}} \{ (1-p)\|X\|^q+\operatorname{Div}(X) \} \\
    &= \inf_{\Omega \setminus \mathcal{F}} \left\{ -\frac{\triangle_p \psi}{\psi^{p-1}} \right\}\cdot
    \end{aligned}
\end{eqnarray}

Note that, on the set where $t$ is smooth, we have $\nabla \psi(x) = \nabla (\omega \circ t)(x)$. By the chain rule, it follows that $\nabla \psi(x) = \omega'(t(x))\,\nabla t(x)$. Therefore, $|\nabla \psi(x)| = |\omega'(t(x))|\cdot |\nabla t(x)|$. Since $|\nabla t(x)| = 1$, it follows that $|\nabla \psi(x)| = |\omega'(t(x))|$.
\begin{eqnarray}\label{divergent}
    \begin{aligned}
      \triangle_p \psi &=  \operatorname{div}(|\nabla \psi|^{p-2}\nabla \psi) \\
      &= \operatorname{div}(|\omega' (t)|^{p-2}\omega'(t)\nabla t) \\
      &= (|\omega' (t)|^{p-2}\omega'(t))' + |\omega' (t)|^{p-2}\omega'(t) \triangle t \\
      &= | \omega(t)|^{p-2} \big( (p-1)\omega''(t) + \omega'(t) \triangle t \big).
    \end{aligned}
\end{eqnarray}
Since the radial sectional curvature satisfies
$K(\partial_t,v) \leq c$ for all $v \perp \partial_t$ with $|v| \leq 1$.
Then, by the Laplacian Comparison Theorem, we have $\triangle t \leq (m-1){S_c'(t)}/{S_c(t)}$ in  $B_M(p_0,r)\setminus \operatorname{Cut}(p_0)$, where $S_c$ denotes the model function of the simply connected space form $\mathbb{M}^m(c)$. Consequently, by Bishop's Theorem together with \eqref{divergent} and \eqref{eqspace2}, it follows that, for $0 < t < d(\xi)$, the following inequality holds
\begin{equation}\label{estimatesup}
\begin{aligned}
   -\dfrac{\triangle_p \psi(\exp_{p_0}(t \xi))}{\psi^{p-1}} &= -\frac{1}{|\omega(t)|} \big( (p-1)\omega''(t) + \omega'(t) \triangle t(\exp_{p_0}(t \xi)) \big) \\
&= -\frac{1}{|\omega(t)|} \left( (p-1)\omega''(t) + \omega'(t) \frac{\sqrt{g(t,\xi)}\,'}{\sqrt{g(t,\xi)}} \right) \\
&\geq  -\frac{1}{|\omega(t)|} \left( (p-1)\omega''(t) + (m-1)\omega'(t) \frac{S_c'(t)}{S_c(t)} \right) \\
&= \lambda_{1,p}^{\mathcal{D}}(B(r)).
\end{aligned}    
\end{equation}
where $g(t,\xi)$ denotes the determinant of the metric tensor expressed in these
coordinates. Thus, using the estimate given in \eqref{Newton-Montenegro-Barnabé} together with \eqref{estimatesup}, it follows that
\begin{equation*}
    \lambda_{1,p}(B_M(p_0,r)) \geq \lambda_{1,p}^{\mathcal{D}}(B(r)).
\end{equation*}

Now assume that equality holds in \eqref{comparison}. In this case, since the
infimum of the Rayleigh quotient is attained, we must, for the same reason as
before, obtain a positive solution $\varphi$ to the Dirichlet eigenvalue
problem.
\begin{equation}\label{eq:p-eigbola}
\begin{cases}
 \triangle_p \varphi =  - \lambda_{1,p}^{\mathcal{D}}(B_M(p_0,r))\,|\varphi |^{p-2}\varphi 
& \text{em } (B_M(p_0,r);\\
\varphi  =  0 & \text{em } \partial (B_M(p_0,r)).
\end{cases}
\end{equation}
In this way, we may apply Theorem \ref{thm:pbarta} with $\varphi = \psi$ and
observe that, for every $t \in (0, d(\xi))$, the following identity
holds:
\begin{equation}\label{IGUAL}
    \lambda_{1,p}(B_M(p_0,r)) =  - \frac{\triangle_p  \psi(\exp_{p_0}(t \xi))}{ \psi^{p-1}}.
\end{equation}
Therefore, \eqref{IGUAL} forces the inequality in \eqref{estimatesup} to become
an equality, and hence we must have that
\begin{equation*}
   \frac{\sqrt{g(t,\xi)}\,'}{\sqrt{g(t,\xi)}} =(m-1)\omega'(t) \frac{S_c'(t)}{S_c(t)} \cdot
\end{equation*}
By Bishop's Theorem, we have $\mathcal{R} = c \cdot \mathrm{Id}$ and
$\mathcal{A} = S_c \cdot \mathrm{Id}$ on the interval $[0,r]$, which yields an
isometry between $B_M(p_0,r)$ and $B(r)$.

\medskip
\noindent\textbf{Claim}. \textit{The vector field $X \in \mathcal{W}^{1,1}(B_M(p_0,r))$}. In fact, using Remark \ref{caracteeizationDiv}, it suffices to show that for every $f \in \mathcal{C}_0^{\infty}(B_M(p_0,r))$ one has $fX \in \mathcal{W}^{1,1}(B_M(p_0,r))$. Since $\mathcal{F} = \{p_0\} \cup \operatorname{Cut}(p_0)$, by Proposition \ref{LemmaSobolev} it is enough to prove that $\operatorname{div}(fX) \in L^{1}(B_M(p_0,r))$ and $fX \in \mathcal{C}^1(B_M(p_0,r)\setminus \mathcal{F}) \cap L^{\infty}(B_M(p_0,r))$. By construction, the vector field $X$ is smooth in $B_M(p_0,r)\setminus \mathcal{F}$. We now follow an argument analogous to that in Bessa--Montenegro \cite{BessaMontenegro}. First, observe that condition \eqref{Field} ensures that $fX \in \mathcal{C}^1(B_M(p_0,r)\setminus \mathcal{F}) \cap L^{\infty}(B_M(p_0,r))$. Moreover, in the set $B_M(p_0,r)\setminus \mathcal{F}$ we have $\operatorname{div}(X) = \operatorname{Div}(X)$ and
$\operatorname{div}(fX) = \langle \nabla f, X \rangle + f\,\operatorname{div}(X)$.
Integrating over $B_M(p_0,r)\cap \operatorname{supp}(f)$ yields
\begin{equation}\label{div}
\begin{aligned}
    \int_{B_M(p_0,r)\cap\operatorname{supp}(f)} \operatorname{div}(fX) \, \rm{d}\mu &= \int_{B_M(p_0,r)\cap\operatorname{supp}(f)} \left[ \langle \nabla f, X \rangle + f\operatorname{div}(X) \right] \, \rm{d}\mu \\
    &\leq \int_{B_M(p_0,r)\cap\operatorname{supp}(f)} |\langle \nabla f, X \rangle|  \, {\rm d}\mu + \int_{B_M(p_0,r)\cap\operatorname{supp}(f)} |f \operatorname{div}(X) |  \, \rm{d}\mu \\
    &\leq \int_{B_M(p_0,r)\cap\operatorname{supp}(f)} |\nabla f| |X|  \, {\rm d}\mu + \int_{B_M(p_0,r)\cap\operatorname{supp}(f)} |f| |\operatorname{div}(X)|  \, \rm{d}\mu. 
\end{aligned}
\end{equation}
Where $d\mu$ denotes the induced Riemannian measure on $B_M(p_0,r)\cap \operatorname{supp}(f)$. The first term in \eqref{div} is controlled by the following standard estimate.
\begin{equation}
    \int_{B_M(p_0,r)\cap\operatorname{supp}(f)} |\nabla f||X|  \, {\rm d}\mu \leq \sup_{\operatorname{supp}(f)}\left\{ |\nabla f| |X| \right\}\mu\left( B_M(p_0,r)\cap\operatorname{supp}(f) \right) < \infty.
\end{equation}
For the second term in \eqref{div}, we note that
\begin{equation}\label{div2}
\begin{aligned}
\int_{B_M(p_0,r)\cap\operatorname{supp}(f)} |f| |\operatorname{div}(X)| \, \rm{d}\mu &\leq  \sup_{\operatorname{supp}(f)}\left\{ |f| \right\} \left( \int_{B_M(p_0,r)\cap\operatorname{supp}(f)} |\operatorname{div}(X)|  \, \rm{d}\mu \right). 
\end{aligned}
\end{equation}
It remains to control the last term in \eqref{div2}. By Remark 2.1 in \cite{LimaMontenegroSantos2010}, we have
\begin{equation}\label{divpsi}
    \begin{aligned}
        \mathrm{div}(X) &= \mathrm{div} \left(- \frac{|\nabla \psi|^{p-2}\nabla \psi}{|\psi|^{p-2}\psi} \right) \\
        &= - \frac{\mathrm{div} \left( |\nabla \psi|^{p-2}\nabla \psi \right)}{\psi^{p-1}} - \left\langle  |\nabla \psi|^{p-2}\nabla \psi, \frac{1}{\psi^{p-1}}  \right\rangle \\
        &= -\frac{\triangle_p \psi}{\psi^{p-1}} + (p-1) \psi^{-p} |\nabla \psi|^{p-2}\langle \nabla \psi, \nabla \psi \rangle \\
        &= -\frac{\triangle_p \psi}{\psi^{p-1}} + (p-1) \frac{|\nabla \psi|^{p}}{\psi^p}.
    \end{aligned}
\end{equation}
Since \( |X(x)| \leq \bigl(|\psi'(t)|^{p-2}\psi'(t)\bigr)\big/\bigl(|\psi(t)|^{p-2}\psi(t)\bigr) < \infty\) for \( x = \exp_{p_0}(t\xi) \in B_M(p_0,r)\cap \operatorname{supp}(f) \), we may repeat the computation in \eqref{estimatesup} and combine it with \eqref{divpsi} to obtain
\begin{equation}\label{divl}
\mathrm{div}(X)  = \frac{1}{|\omega(t)|} \left( (1-p)\omega''(t) + (p-1)\dfrac{(\omega'(t))^p}{(\omega(t))^p} - \omega'(t) \frac{\sqrt{g(t,\xi)}\,'}{\sqrt{g(t,\xi)}} \right).
\end{equation}
Now, we integrate \eqref{divl} over $B_M(p_0,r)\cap \operatorname{supp}(f)$.
\begin{equation}\label{divbounded}
\begin{aligned}
         \int_{B_M(p_0,r)\cap\operatorname{supp}(f)} |\operatorname{div}(X)|  \, \rm{d}\mu  \leq  (p-1)&\int_\xi \int_0^{t(\xi)} \left( \left |\dfrac{\omega''}{\omega}\right| +  \left |\dfrac{(\omega')^p}{\omega^p}\right|\right)\sqrt{g(t,\xi)} dtd\xi \\
         +& \int_\xi \int_0^{t(\xi)} \left| \frac{\omega'}{\omega} \right| \frac{\sqrt{g(t,\xi)}\,'}{\sqrt{g(t,\xi)}} dtd\xi. 
         \end{aligned}
\end{equation}
Where $t(\xi)$ denotes the largest $t < \min\{d(\xi), r\}$ such that
$\exp_{p_0}(t\xi) \in \operatorname{supp}(f)$. In this case, note that the equation \eqref{divbounded} is uniformly bounded. Moreover, substituting \eqref{divbounded} into \eqref{div2}, we obtain the desired bound for the second term in \eqref{div}.
\end{proof}

\subsection{Proof of Theorem \ref{thm:pminsub_ball}}
\medskip

\noindent
\textbf{Theorem \ref{thm:pminsub_ball}} \textit{
Let $(N^n,g)$ be an $n$-dimensional Riemannian manifold and fix $p_0\in N$. Assume that the radial sectional curvatures of $N$ with respect to $p_0$ satisfy $K_N(\partial_t,v) \leq c$ for every $x\in B_N(p_0,r)\setminus \mathrm{Cut}(p_0)$ and every unit vector 
$v\perp \partial_t$. Suppose moreover that $\mathcal{H}^{n-1}\!\big(\mathrm{Cut}(p_0)\cap B_N(p_0,r)\big)=0$. Let $\varphi:M^m\hookrightarrow N^n$ be a minimal immersion and let 
$\Omega\subset \varphi^{-1}(B_N(p_0,r))$ be a connected component. 
Assume that the extrinsic distance function $t(x)=\operatorname{dist}_N(\varphi(x),p_0)$ satisfies $|\nabla^M t|\ge k>0$ in $\Omega$. Let $B(r)\subset \mathbb{M}^m(c)$ denote the geodesic ball of radius $r$ in the 
$m$-dimensional simply connected space form of constant curvature $c$. 
Then for every $2\le p<\infty$ and $r\le r_\ast(c)$, it holds that
\begin{equation}\label{ppBessaMontenegro}
\lambda_{1,p}(\Omega) \geq k^{p-2}\,\lambda_{1,p}^{\mathcal D}(B(r)).    \end{equation}
If equality holds for some bounded $\Omega$, then necessarily 
$k=1$ and $\Omega$ is isometric to $B(r)$.}

\begin{proof}
Let $\mathbb{M}^m(c)$ be the $m$-dimensional model space of constant curvature $c$, with base point $x_0$ and radial coordinate given by $t(x):=\operatorname{dist}_{\mathbb{M}^m(c)}(x,x_0).$
Let $B(r)$ denote the geodesic ball of radius $r$ centered at $x_0$. Consider the eigenvalue problem associated with the model space $\mathbb{M}^m(c)$
\begin{equation}\label{eq:p-eig-modelo2}
\begin{cases}
 \triangle_p u =  - \lambda_{1,p}^{\mathcal{D}}(B(r))\,|u |^{p-2}u 
& \text{em } B(r);\\
u  =  0 & \text{em } \partial B(r).
\end{cases}.
\end{equation}
Thus, using the same argument as in the proof of Theorem \ref{thm:Cheng}, there
exists a function $\omega\colon [0,r] \to \mathbb{R}$ such that
$u(x) = \omega(t(x))$. Furthermore, $\omega > 0$ in $(0,r)$, $\omega(r) = 0$,
and $\omega'(t) \leq 0$ e $\omega'(t) = 0$ iff $t=0$. Let $u(x) = \omega(t(x)) > 0$ be normalized, i.e, $\omega(0)=1$ be defined in $\mathbb{M}^m(c)$.
Therefore, the eigenvalue equation $\triangle_p u + \lambda_{1,p}^{\mathcal{D}}(B(r))\,|u|^{p-2}u = 0$ can be written as in \eqref{eq:eigenp}, that is, for all $t \in [0,r]$ we have
\begin{equation}\label{eq:eigenp2}
   \left( |\omega'(t)|^{p-2} \omega'(t) \right)' + (m-1) \frac{S_{c}'(t)}{S_{c}(t)} |\omega'(t)|^{p-2} \omega'(t) + \lambda_{1,p}^{\mathcal{D}}(B(r))  (\omega(t))^{p-1}=0.
\end{equation}
For each $\xi \in T_{p_0}N$ with $|\xi|=1$, define $d(\xi)>0$ to be the largest real number such that the geodesic $\gamma_\xi(t)=\exp_{p_0}(t\xi)$ minimizes the distance between $\gamma_\xi(0)=p_0$ and $\gamma_\xi(t)$ for all $t \in [0,d(\xi)]$ (with the possibility that $d(\xi)=\infty$). Under these assumptions, we have that
\[ B_N(r) \setminus \mathrm{Cut}(p_0) = \exp_{p_0} ( \{ t \xi \in T_{p_0}N; \,  0 \leq t < \min \{r,\text{d}\xi \}; \ | \xi |=1 \} ). \]
Consider the function $\overline{\omega} \colon B_N(p_0,r) \to [0,\infty)$ defined by
\begin{equation}\label{modelfunction}
\overline{\omega}(x) \coloneqq \begin{cases}
\omega(t) & \text { if } x=\exp_{p_0}(t \xi), \ t \in[0, d(\xi)) \cap[0, r]; \\
0 & \text { if } x=\exp_{p_0}(d(\xi)\xi) \in \operatorname{Cut}(p_0).
\end{cases}
\end{equation}
In this setting, let $\psi \colon \Omega \to \mathbb{R}$ be defined by $\psi \coloneqq \overline{\omega} \circ \varphi$, where $\varphi:M \hookrightarrow N$ is a minimal immersion. Consider the vector field $X$ given by
\[ X \coloneqq - \frac{|\nabla \psi|^{p-2}\nabla \psi}{|\psi|^{p-2}\psi} \cdot \]
Note that the vector field $X$ is not smooth precisely in the set where the distance function from $p$ is not smooth, that is, $\mathcal{F} = \varphi^{-1}(\operatorname{Cut}_N(p_0))$. By hypothesis, we have $\mathcal{H}^{n-1}\bigl( \Omega \cap \mathcal{F} \bigr)=0$.
Due to the regularity of $\psi$, it follows a fortiori that $X \in \mathcal{C}^1(\Omega \setminus \mathcal{F}) \cap L^{\infty}(\Omega)$. Furthermore, the divergence of $X$ can be explicitly computed as in \eqref{divpsi}. As a consequence, we have $\operatorname{div}(X) \in L^{1}(\Omega)$ it then follows from \cite[Lemma 3.1]{BessaMontenegro} that $X \in \mathcal{W}^{1,1}(\Omega)$, and that the weak divergence $\operatorname{Div}(X)$ coincides with the classical divergence $\operatorname{div}(X)$ in $\Omega \setminus \mathcal{F}$.
Using Theorem \ref{Thm:BMS} together with Remark \ref{rmk:BMS}, we obtain the following estimate
\begin{equation}\label{pbartabessamontenegro}
\begin{aligned}
    \lambda_{1,p}(\Omega) &\geq \inf_{\Omega} \{ (1-p)\|X\|^q+\operatorname{Div}(X) \} \\
    &= \inf_{\Omega \setminus \mathcal{F}} \{ (1-p)\|X\|^q+\operatorname{Div}(X) \} \\
    &= \inf_{\Omega \setminus \mathcal{F}} \left\{ -\frac{\triangle_p \psi}{\psi^{p-1}}  \right\}.
    \end{aligned}
\end{equation}
Let $\varphi \colon M \hookrightarrow N$ be an isometric minimal immersion of a complete Riemannian $m$-manifold $M$ into a complete Riemannian $n$-manifold $N$ with sectional curvature $K_N \leq c \leq 0$. Let $\Omega \subset \varphi^{-1}(B(r))$ be a connected component and let $\psi = \overline{\omega} \circ \varphi \colon \Omega \to [0,\infty)$, where $\overline{\omega} \colon B(r) \to [0,\infty)$ denotes the transplanted function on $N$, as defined in \eqref{modelfunction}. Our purpose is to compute the $p$-Laplacian $\triangle_p \psi$. Let $\{e_1,\dots,e_m\}$ be a local orthonormal frame of $T_xM$ and we identify $e_i$ with $d\varphi_x(e_i)\in T_{\varphi(x)}N$. In this setting, we obtain on $\Omega\setminus\mathcal F$ that
\begin{equation}
\begin{aligned}\label{Jorge-Kautrofiots}
    \triangle_p \psi &= \operatorname{div} \left( |\nabla \psi|^{p-2} \nabla \psi  \right) \\
&= \sum_{i=1}^m \left\langle \nabla_{e_i} |\nabla \psi|^{p-2}\,\nabla \psi, e_i \right\rangle \\
&= \sum_{i=1}^m  \left[ \overbrace{ \left\langle  \nabla |\nabla \psi|^{p-2}, e_i \right\rangle}^{=\, e_i \left( |\nabla \psi|^{p-2}\right)} \left\langle \nabla \psi, e_i \right\rangle + |\nabla \psi|^{p-2} \left\langle \nabla_{e_i}\nabla \psi, e_i \right\rangle \right] \\
&= |\nabla \psi|^{p-2} \triangle \psi + \sum_{i=1}^m \left\langle  \nabla |\nabla \psi|^{p-2}, e_i \right\rangle \left\langle \nabla \psi, e_i \right\rangle \\
&= |\nabla \psi|^{p-2} \triangle \psi + \left\langle  \nabla |\nabla \psi|^{p-2},
\nabla \psi \right\rangle \\
&= |\nabla \psi|^{p-2} \triangle \psi + (p-2)|\nabla \psi|^{p-3} \left\langle 
\nabla \psi, \nabla |\nabla \psi| \right\rangle.
\end{aligned}
\end{equation}
Moreover, on the set $\{|\nabla\psi|>0\}$ one has 
\[\left\langle \nabla\!\left(|\nabla \psi|^{p-2}\right), \nabla \psi \right\rangle
=(p-2)|\nabla \psi|^{p-4}\operatorname{Hess}\psi(\nabla \psi, \nabla \psi),\] 
and therefore substituting this identity into \eqref{Jorge-Kautrofiots}, we obtain 
\begin{equation}\label{jorgekaushess}
\Delta_p \psi = |\nabla \psi|^{p-2}\Delta \psi
+(p-2)|\nabla \psi|^{p-4}\operatorname{Hess}\psi(\nabla \psi, \nabla \psi).   
\end{equation}
Hence, the above identity holds almost everywhere in $\Omega\setminus\mathcal F$. On the other hand, $\Delta\psi$ can be computed via the Jorge--Koutrofiotis formula
\cite[p. 13]{JorgeKoutrofiotis1981} (see also \cite{BessaMontenegro2003}). Since $\psi=\overline{\omega}\circ\varphi$, we obtain
\begin{equation}\label{eq:JK-Laplacian}
\Delta\psi(x) = \sum_{i=1}^m \operatorname{Hess}_N \overline{\omega}\big(\varphi(x)\big)(e_i,e_i)
+ \left\langle \nabla^N \overline{\omega}\big(\varphi(x)\big),\sum_{i=1}^m \mathrm{I}\mathrm{I}(e_i,e_i)\right\rangle,
\end{equation}
where $\mathrm{I}\mathrm{I}$ denotes the second fundamental form of $\varphi$ and $\{e_1,\dots,e_m\}$ is a local orthonormal frame on $M$
(identified with $d\varphi_x(e_i)\in T_{\varphi(x)}N$). Since $\varphi$ is minimal immersion, its mean curvature vector vanishes, hence $\sum_{i=1}^m \mathrm{I}\mathrm{I}(e_i,e_i)=mH=0$.
Therefore, substituting \eqref{eq:JK-Laplacian} into \eqref{jorgekaushess}, we get 
\begin{equation}\label{JorgeKausp}
    \begin{aligned}
        \triangle_p \psi(x) &= |\nabla \psi|^{p-2} \left( \sum_{i=1}^m \operatorname{Hess}_N \overline{\omega}\big(\varphi(x)\big)(e_i,e_i) + \stackrel{=0}{\left\langle \nabla^N \overline{\omega}\big(\varphi(x)\big) , \sum_{i=1}^m \mathrm{I}\mathrm{I}(e_i,e_i)  \right\rangle}  \right)  \\ &+ (p-2)|\nabla \psi|^{p-4}\operatorname{Hess}\psi(\nabla \psi, \nabla \psi) \\
  &= |\nabla \psi|^{p-2} \sum_{i=1}^m \operatorname{Hess}_N \overline{\omega}\big(\varphi(x)\big)(e_i,e_i) + (p-2)|\nabla \psi|^{p-4}\operatorname{Hess}\psi(\nabla \psi, \nabla \psi). 
    \end{aligned}
\end{equation}
Assume $\varphi(x)=\exp_{p_0}(t\xi)$ and set $t=d_N(p_0,\varphi(x))$.
Let $\{e_1,\ldots,e_m\}$ be an orthonormal basis of $T_xM$, identified with $d\varphi_x(e_i)\in T_{\varphi(x)}N$. Choose the basis so that $e_2,\ldots,e_m$ are tangent to the geodesic sphere
$\partial B_N(p_0,t)$ and $e_1=\cos(\alpha(x))\,\partial_t+\sin(\alpha(x))\,\partial_\theta$,
where $\partial_\theta\in\mathrm{span}\{e_2,\ldots,e_m\}$, $|\partial_\theta|=1$, and
$\langle \partial_\theta,\partial_t\rangle=0$.
Since $\overline{\omega}=\omega\circ t$, we have $\operatorname{Hess}_N\overline{\omega}=\omega''(t)\,dt\otimes dt+\omega'(t)\operatorname{Hess}_N(t)$ on $N\setminus \mathrm{Cut}(p_0)$. Therefore, proceeding analogously to \cite[p. 12]{BessaMontenegro}, we can compute $\triangle\psi$ for $x\in\Omega\setminus\mathcal F$, obtaining the following expression for the Laplacian of $\psi$.
\begin{equation}\label{Laplacianoimersão}
\begin{aligned}
\triangle \psi(x)
&= \sum_{i=1}^m  \operatorname{Hess}_N \overline{\omega}\big(\varphi(x)\big)(e_i,e_i) \\
&= \omega''(t)\,\bigl(1 - \sin^2 \left( \alpha(x) \right) \bigr) \\
&\quad + \omega'(t) \sin^2 \left( \alpha(x) \right)
   \operatorname{Hess}_N(t)\!\left(\frac{\partial}{\partial\theta},
   \frac{\partial}{\partial\theta}\right) \\
&\quad + \omega'(t)\sum_{i=2}^m 
   \operatorname{Hess}_N(t)(e_i, e_i).
\end{aligned}
\end{equation}
Substituting \eqref{Laplacianoimersão} into \eqref{JorgeKausp}, we compute \eqref{jorgekaushess} for $x\in \Omega\setminus\mathcal F$. Since $\psi=\omega(t)$ and $\nabla\psi=\omega'(t)\nabla^M t$,
we have $|\nabla\psi|=|\omega'(t)|\,|\nabla^M t|
=|\omega'(t)|\cos\alpha(x)$. Moreover, $\operatorname{Hess}\psi(\nabla\psi,\nabla\psi)
=(\omega'(t))^2\cos^2\alpha(x)\, \operatorname{Hess}\psi(e_1,e_1)$, and
$\operatorname{Hess}\psi(e_1,e_1) = \omega''(t)\cos^2\alpha(x) + \omega'(t)\sin^2\alpha(x) \operatorname{Hess}_N(t)(\partial_\theta,\partial_\theta)$. Therefore,
\begin{equation}
\begin{aligned}
     \triangle_p \psi &= (p-1)\cos^{p-2}{(\alpha(x))}|\omega'(t)|^{p-2}\omega''(t)\left(1 - \sin^2{\alpha(x)} \right) \\
     &+(p-1)|\omega'(t)|^{p-2}\cos^{p-2}{(\alpha(x))}\,\omega'(t) \sin^2 \left( \alpha(x) \right) \operatorname{Hess}_N(t)\!\left(\frac{\partial}{\partial\theta},
   \frac{\partial}{\partial\theta}\right) \\
   &+ |\omega'(t)|^{p-2}\cos^{p-2}{(\alpha(x))}\, \omega'(t)\sum_{i=2}^m 
   \operatorname{Hess}_N(t)(e_i, e_i). 
\end{aligned}
\end{equation}
Now we will proceed as follows, we will add and subtract appropriate terms to achieve the desired model space equality, more precisely by adding and subtracting the terms $(m-1)\bigl(S_c'(t)/S_c(t)\bigr)|\omega'(t)|^{p-2}\omega'(t)\cos^{p-2}{\left(\alpha(x)\right)}$, $(p-1)\bigl(S_c'(t)/S_c(t)\bigr)|\omega'(t)|^{p-2}\omega'(t)\sin^2{\left(\alpha(x)\right)}\cos^{p-2}{\left(\alpha(x)\right)}$ we obtain  for $t\in(0,r)$  the following expression
\begin{equation}\label{pLaplacianoimersão}
\begin{aligned}
\triangle_p \psi(x)
&=  \cos^{p-2}{\left(\alpha(x)\right)}\left[ \Bigl(|\omega' \bigl(t \bigr)|^{p-2} \omega' \bigl(t\bigr)\Bigr)' +  (m-1)\frac{S'_c(t)}{S_c(t)} |\omega' \bigl(t\bigr)|^{p-2} \omega' \bigl(t\bigr) \right] \\
&\quad + \sum_{i=2}^m \left( \operatorname{Hess}_N(t)(e_i, e_i) - \frac{S'_c(t)}{S_c(t)}\right)|\omega' (t)|^{p-2} \omega'(t)\cos^{p-2}{\left(\alpha(x)\right)} \\
&\quad + (p-1)\left(\frac{S'_c(t)}{S_c(t)}\omega'(t) - \omega''(t) \right)  |\omega'(t)|^{p-2} \sin^2\! \alpha(x)\cos^{p-2}{\left(\alpha(x)\right)}  \\
& \quad + (p-1)\left( \operatorname{Hess}_N(t) \left( \dfrac{\partial}{\partial \theta}, \dfrac{\partial}{\partial \theta} \right) - \frac{S'_c(t)}{S_c(t)}\right) |\omega' (t)|^{p-2}\omega'(t) \sin^2\! \alpha(x)\cos^{p-2}{\left(\alpha(x)\right)}.
\end{aligned}
\end{equation}
Observe that the first term in \eqref{pLaplacianoimersão} can be replaced by means of the model space identity \eqref{eq:eigenp2}. Consequently, on the set $\{ x \in \Omega \,;\, |\nabla \psi(x)| > 0 \}$, we have $\cos(\alpha(x)) > 0$. Therefore, it follows that
\begin{equation}\label{pLaplacianoimersão2}
\begin{aligned}
-\frac{\triangle_p \psi(x)}{(\psi(x))^{p-1}\cos^{p-2}{\left(\alpha(x)\right)}}
&= \lambda_{1,p}^{\mathcal{D}}(B(r)) \\
&\quad - \sum_{i=2}^m \left( \operatorname{Hess}_N(t)(e_i, e_i) - \frac{S'_c(t)}{S_c(t)}\right) |\omega' (t)|^{p-2} |\omega(t)|^{1-p} \omega' (t) \\
&\quad - (p-1)\left( \frac{S'_c(t)}{S_c(t)}\omega'(t) - \omega''(t) \right)  |\omega'(t)|^{p-2} |\omega(t)|^{1-p} \sin^2\! \alpha(x) \\
& \quad - (p-1)\left( \operatorname{Hess}_N(t) \left( \dfrac{\partial}{\partial \theta}, \dfrac{\partial}{\partial \theta} \right) - \frac{S'_c(t)}{S_c(t)}\right) |\omega' (t)|^{p-2}|\omega(t)|^{1-p} \omega'(t) \sin^2\! \alpha(x).
\end{aligned}
\end{equation}
Our goal is to show that the terms after $\lambda_{1,p}^{\mathcal D}(B(r))$ are nonnegative. By assumption, the radial sectional curvature of $N$ with respect to $p_0$ satisfies $\mathrm K_N(\partial_t,v)\le c$ for all $x\in B_N(p_0,r)\setminus\mathrm{Cut}(p_0)$ and all unit vectors $v\perp\partial_t$. Hence, by the Hessian Comparison Theorem (see Lema \ref{lem:hessian-comparison}), we obtain
\[
\operatorname{Hess}_N(t(x))(v,v)-\frac{S_c'(t)}{S_c(t)}\ge 0,
\]
for every $v\perp\partial_t$, where $t=t(x)$ and $x=\exp_{p_0}(t\xi)$. Under these assumptions, and recalling that $\omega'(t)\le 0$, we have
\begin{equation}
\begin{aligned}
\sum_{i=2}^m \left( \operatorname{Hess}_N(t)(e_i, e_i) - \frac{S'_c(t)}{S_c(t)}\right) |\omega' (t)|^{p-2} |\omega(t)|^{1-p} \omega' (t) &\leq 0. \\ \left( \operatorname{Hess}_N(t) \left( \dfrac{\partial}{\partial \theta}, \dfrac{\partial}{\partial \theta} \right) - \frac{S'_c(t)}{S_c(t)}\right) |\omega' (t)|^{p-2}|\omega(t)|^{1-p}\omega'(t) \sin^2\! \alpha(x) &\leq 0.
\end{aligned}
\end{equation}
It remains to verify the validity of the following inequality
\begin{equation}\label{inequality2}
\left( \frac{S'_c(t)}{S_c(t)}\omega'(t) - \omega''(t) \right)  |\omega'(t)|^{p-2} |\omega(t)|^{1-p} \sin^2\! \alpha(x) \leq 0.
\end{equation}
To prove this, observe that \eqref{inequality2} imposes constraints on the model space equation \eqref{eq:eigenp2}. More precisely, in order to guarantee \eqref{inequality2}, it suffices to prove that the following inequality holds for all $t\in(0,r)$
\begin{equation}\label{restricmodel}
        \left( p+m-2 \right)\frac{S'_c(t)}{S_c(t)}|\omega'(t)|^{p-2}\omega'(t) + \lambda_{1,p}^{\mathcal{D}}(B(r))(\omega(t))^{p-1} \leq 0.
\end{equation}
Where, $B(r)$ denotes a geodesic ball in the model space $\mathbb{M}^m(c)$. Without loss of generality, we may assume that $c\in\{-1,0,1\}$. For simplicity, we begin with the flat model space, that is, the case $c=0$.

\medskip
\noindent\textbf{Case 1: The flat model space ($c=0$)}. In this setting, substituting the model function given in \eqref{espaceformfunc} with $c=0$ into the equation \eqref{restricmodel}, we obtain the following result.
\begin{equation}\label{restricmodel1}
        \left( p+m-2 \right)\frac{1}{t}|\omega'(t)|^{p-2}\omega'(t) + \lambda_{1,p}^{\mathcal{D}}(B(r))(\omega(t))^{p-1} \leq 0.
\end{equation}
In this case, we begin by observing that
\begin{equation}\label{flatine1}
    \left( t^{m-1} \left| \omega'(t) \right|^{p-2}\omega'(t) \right)' + \lambda_{1,p}^{\mathcal{D}}(B(r))t^{m-1}(\omega(t))^{p-1}=0.
\end{equation}
Define the function $\mathcal{V}_0\colon (0,r)\to\mathbb{R}$ by
\begin{equation*}
\mathcal{V}_0(t) \coloneqq \exp{  \left\{- \lambda_{1,p}^{\mathcal{D}}(B(r))\frac{t^2}{2(p+m-2)}  \right\}}.  
\end{equation*}
Observe that the function $\mathcal{V}_0$ satisfies the following equation
\begin{equation}\label{flatine2}
    \left( t^{m-1} \mathcal{V}_0'(t) \right)' + \left(\frac{m}{p+m-2} \right)\lambda_{1,p}^{\mathcal{D}}(B(r))t^{m-1}\left( 1 - \frac{\lambda_{1,p}^{\mathcal{D}}(B(r))t^2}{m(p+m-2)} \right)\mathcal{V}_0(t)=0.
\end{equation}
Multiplying equation \eqref{flatine1} by $\mathcal{V}_0(t)$ and equation \eqref{flatine2} by $-(\omega(t))^{p-1}$, summing the resulting expressions, and integrating from $0$ to $t$, we obtain the following estimate for all $t \in(0,r)$.
\begin{equation}\label{flateq}
    \begin{aligned}
    t^{m-1} \left( \mathcal{V}_0'(t)\left( \omega(t) \right)^{p-1} - \mathcal{V}_0(t)|\omega'(t)|^{p-2}\omega'(t) \right) &=  \left( \frac{(p-2) \lambda_{1,p}^{\mathcal{D}}(B(r)) }{p+m-2} \right) \int_{0}^{t} s^{m-1}(\omega(s))^{p-1}\mathcal{V}_0(s)ds \\
      &+  \left( \frac{\lambda_{1,p}^{\mathcal{D}}(B(r))}{(p+m-2)} \right)^2 \int_{0}^{t} s^{m+1}(\omega(s))^{p-1}\mathcal{V}_0(s)ds \\
      &+ \int_{0}^{t}s^{m-1}\omega'(s) \underbrace{\left[ (p-1)|\omega(s)|^{p-2} - |\omega'(s)|^{p-2}  \right]}_{\coloneq G(s)}\mathcal{V}'_0(s)ds.
    \end{aligned}
\end{equation}
Studying the sign of the right-hand side of \eqref{flateq} reduces to determining the sign of the following function
\begin{equation*}
   \Psi_0(t) \coloneq \int_{0}^{t} s^{m-1}  \underbrace{\left[ (p-2)|w(s)|^{p-1} + \dfrac{\lambda_{1,p}^{\mathcal{D}}(B(r))}{p+m-2}s^2|\omega(s)|^{p-1} + (p-1)|\omega(s)|^{p-2}|\omega'(s)| - |\omega'(s)|^{p-1} \right]}_{\coloneq \Phi_0(s)}\mathcal{V}_0(s)ds. 
\end{equation*}
Since $\Psi_0'(t)=t^{m-1}\Phi_0(t)\mathcal{V}_0(t)$ and $t^{m-1}\mathcal{V}_0(t)>0$ for all $t\in(0,r)$, the behavior of $\Psi_0$ is entirely determined by the sign of $\Phi_0$. Note that, in the case $p=2$, as shown in \cite[Theorem 1.10]{BessaMontenegro}, one has
\[
\Phi_0(s)=\frac{\lambda_{1,2}^{\mathcal{D}}(B(r))}{m}\,s^{2}\,|\omega(s)|.
\]
In particular, $\Phi_0(s)>0$ for all $s\in(0,r)$, which guarantees that $\Psi$ is positive throughout the interval $(0,r)$. We are therefore left to consider the case $p\in(2,\infty)$. Recall that since $\omega$ is the first Dirichlet $p$-eigenfunction in $B(r)$, it belongs to $\mathcal{C}^{1,\alpha}(B(r))$ for some $\alpha\in(0,1)$ (see \cite[Lemma 5]{le}). Since $\omega$ is normalized so that $\omega(0)=1$, we obtain $\Phi_0(0)=p-2>0$. Moreover, by Hopf's boundary point lemma (see \cite[Theorem 5]{Vazquez}), we have $\Phi_0(r)=-|\omega'(r)|^{p-2}<0$. Therefore, by continuity, the function $\Phi_0$ admits at least one zero in $(0,r)$. For $c=0$, we define the real number $r_\star(c)\in(0,r)$ by
\begin{equation}
r_\star(0) \coloneq \inf \left\{ t>0 \; ; \; \int_{0}^{t} s^{m-1} \Phi_0(s) \mathcal{V}_0(s) ds \leq 0 \right\}.
\end{equation}
The figure below depicts the behavior of the function $\Psi_0$ and the procedure used to define the critical radius $r_\star$.
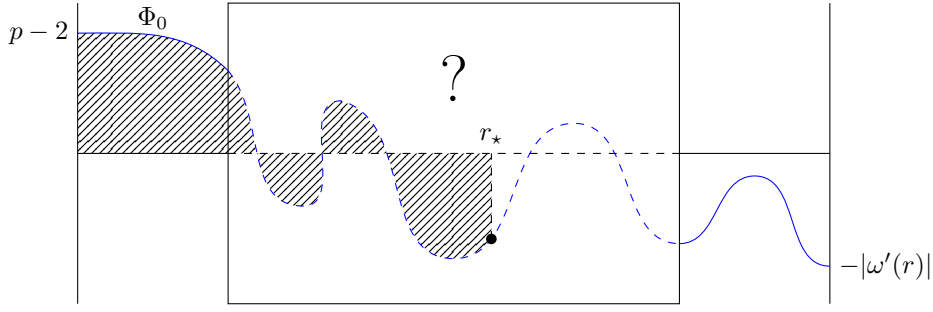
\begin{figure}[htb!]
    \centering
\begin{tikzpicture}[scale=1,>=stealth]
		\def\u{5.0}
		\def\v{2.0}
		\def\l{3.0}
		\def\k{2.0}

		\path[draw,name path=l1] (-\u,0) -- (-\l,0);
		\path[draw,name path=l2,dashed] (-\l,0) -- (\l,0);
		\path[draw,name path=l3] (\l,0) -- (\u,0);
		
		\path[draw,name path=l4] (-\u,-\v) -- (-\u,\v);
		\path[draw,name path=l5] (\u,-\v) -- (\u,\v);
		\path[name path=aux1] (0.5,-\v) -- (0.5,\v);
		\path[draw,name path=q1] (-\l,-\k) rectangle (\l,\k);

		\path[draw,name path=curva1,blue] (-\u,1.6) to[out=0] (-\l,1.1);
		
		\path[draw,name path=curva2,blue,dashed] (-\l,1.1) to[out=-50,in=180] (-\l+1,-0.7) to[out=0,in=180] (-\l+1.5,0.7) to[out=-15,in=180] (-\l+3,-1.4)to[out=0,in=180] (-\l+4.6,0.4)to[out=0,in=180] (\l,-1.2);
		
		\path[draw,name path=curva3,blue] (\l,-1.2) to[out=0,in=180] (\l+1,-0.3) to[out=0,in=180] (\u,-1.5);

		\path[name intersections={of=aux1 and curva2}];
		\coordinate (p0) at (intersection-1);
		\path[draw,name path=l6,dashed] (p0) -- (p0|-\u,0) node[above] {$r_{\star}$};

		\node[left] (t1) at (-\u,1.6) {$p-2$}; 	
		\node[] (t2) at (0,1) {{\Huge ?}}; 
		\node[right] (t3) at (\u,-1.5) {$-|\omega'(r)|$};

		\path[name intersections={of=l2 and curva2}];
		\coordinate (p1) at (intersection-1);
		\coordinate (p2) at (intersection-2);
		\coordinate (p3) at (intersection-3);
		
		\path[pattern=north east lines] (-\u,1.6) to[out=0] (-\l,1.1)  to[out=-55,in=105] (p1) -- (-\u,0) ;
		
		\path[pattern=north east lines] (p1) to[out=-80,in=185] (-\l+1,-0.7) to[out=0,in=-88] (p2);
		
		\path[pattern=north east lines] (p2) to[out=90,in=180] (-\l+1.5,0.7) to[out=-20,in=110] (p3);
		
		\path[pattern=north east lines] (p3) to[out=-70,in=180] (-\l+3,-1.4) to[out=10,in=-126] (p0) -- (p0|-\u,0) -- cycle;
		
		\fill (p0) circle (2pt);
	
		\node (tk) at (-4,1.8) {$\Phi_0$};

	\end{tikzpicture}

    \caption{Typical behavior of the function $\Phi_0(s)$.}
    \label{fig:placeholder}
\end{figure}

\noindent In this case, for all $t\in(0,r_\star)$, we have $\mathcal{V}_0'(t)\,\omega(t)^{p-1}-\mathcal{V}_0(t)\,|\omega'(t)|^{p-2}\omega'(t)\geq 0$, which guarantees the inequality required in \eqref{restricmodel1}.

\medskip
\noindent\textbf{Case 2: The spherical model space ($c=1$).} In this case, by substituting $c=1$ into the model function \eqref{espaceformfunc}, in order to prove \eqref{inequality2} it suffices to show that the following inequality holds for all $t\in(0,\pi/2)$.

\begin{equation}\label{restricmodelshpere}
        \left( p+m-2 \right)\frac{S'_1(t)}{S_1(t)}|\omega'(t)|^{p-2}\omega'(t) + \lambda_{1,p}^{\mathcal{D}}(B(r))(\omega(t))^{p-1} \leq 0.
\end{equation}
First, we note that
\begin{equation}\label{modelsheric1}
    \left( S_1^{m-1}(t)|\omega'(t)|^{p-2}\omega'(t)\right)' +\lambda_{1,p}^{\mathcal{D}}(B(r)) S_1^{m-1}(t)(\omega(t))^{p-1} = 0.
\end{equation}
We define $\mathcal{V}_1\colon(0,\pi/2)\to\mathbb{R}$ by
\begin{equation*}
    \mathcal{V}_1(t) \coloneqq  \left( S_1'(t) \right)^{-\frac{\lambda_{1,p}^{\mathcal{D}}(B(r))}{p+m-2}}.  
\end{equation*}
After some algebraic manipulations, we obtain the following identity
\begin{equation}\label{modelsheric2}
    \left(S_1^{m-1}(t)\mathcal{V}'_1(t) \right)' - \lambda_{1,p}^{\mathcal{D}}(B(r))S_1^{m-1}(t)\left[ \frac{m}{p+m-2} + \frac{1}{p+m-2} \left| \frac{S_1(t)}{S_1'(t)} \right|^{2} \left( 1 + \dfrac{\lambda_{1,p}^{\mathcal{D}}(B(r))}{p+m-2} \right)  \right]\mathcal{V}_1(t)=0.
\end{equation}
Multiplying \eqref{modelsheric1} by $\mathcal{V}_1(t)$ and \eqref{modelsheric2} by $-(\omega(t))^{p-1}$, we arrive at the following expression
\begin{equation}\label{identitysheprical}
    \begin{aligned}
    S_1^{m-1}(t) \left( \mathcal{V}'_1(t)\left( \omega(t) \right)^{p-1} - \mathcal{V}_1(t)|\omega'(t)|^{p-2}\omega'(t) \right) &= \int_{0}^{t} S_1^{m-1}(s)\omega'(s) \left[ (p-1)|\omega(s)|^{p-2} - |\omega'(s)|^{p-2}  \right] \mathcal{V}_1'(s)ds \\
    &+C_1(m,p) \int_{0}^{t} S_1^{m-1}(s)(\omega(s))^{p-1}\mathcal{V}_1(s)ds \\
    &+C_2(m,p) \int_{0}^{t}  S_1^{m-1}(s) (\omega(s))^{p-1} \left|\frac{S_1(s)}{S_1'(s)}\right|^2\mathcal{V}_1(s)ds.
    \end{aligned}
\end{equation}
Where, $C_1(m,p)$ and $C_2(m,p)$ are the constants defined by
\begin{equation}\label{CK}
    C_1(m,p) \coloneq  \lambda_{1,p}^{\mathcal{D}}(B(r)) \left( \frac{p+2m-2}{p+m-2} \right) \quad \text{e} \quad C_2(m,p) \coloneq  \lambda_{1,p}^{\mathcal{D}}(B(r))\left( \frac{1}{p+m-2} + \frac{\lambda_{1,p}^{\mathcal{D}}(B(r))}{(p+m-2)^2}\right).
\end{equation}
Note that, by substituting $p=2$, the right-hand side of \eqref{identitysheprical} is always positive. For $p \in (2,\infty)$, we can rewrite the right-hand side of \eqref{identitysheprical} as follows:
\begin{equation*}
       \Psi_1(t) \coloneq \int_{0}^{t}  S_1^{m-1}(s)  \underbrace{ \left[ \left( C_1 + C_2\left|\frac{S_1(s)}{S_1'(s)}\right|^2\right) |\omega(s)|^{p-1} + C_3\left|\frac{S_1(s)}{S_1'(s)}\right| \left( \frac{|\omega'(s)|^{p-1}}{p-1}- |\omega(s)|^{p-2}|\omega'(s)|  \right) \right]}_{\coloneq \Phi_1(s)}\mathcal{V}_1(s)ds 
\end{equation*}
where $C_3(m,p)=\lambda_{1,p}^{\mathcal{D}}(B(r))/(p+m-2)$. 
Since $\Psi_1'(t)=S_1^{m-1}(t)\Phi_1(t)\mathcal{V}_1(t)$ and $S_1^{m-1}(t)\mathcal{V}_1(t)>0$ for all $t\in(0,r)$ with $r<\pi/2$, the behavior of $\Psi_1(t)$ on $(0,r)$ is completely determined by $\Phi_1(t)$. As observed previously, $\omega(0)=1$, $\omega'(0)=0$, and $S_1(0)=0$, so it follows that $\Phi(0)=C_1(m,p)>0$. By Hopf's Lemma, $\omega'(r)<0$, and thus 
\[\Phi_1(r)=C_3(m,p)\,\frac{S_1(r)}{S_1'(r)}\,\frac{1}{p-2}|\omega'(r)|^{p-2}>0.\]
Now, since $p>2$. Then Young's inequality can be applied in the following form
\begin{equation*}
    |\omega(t)|^{p-2}|\omega'(t)| \leq \dfrac{p-2}{p-1}|\omega(t)|^{p-1} + \frac{1}{p-1}|\omega'(t)|^{p-1}.
\end{equation*}
In this way, we obtain the following bound for the nonlinear term of $\Psi_1$
\begin{equation*}
  \frac{1}{p-1}|\omega'(t)|^{p-1} -  |\omega(t)|^{p-2}|\omega'(t)| \geq -\dfrac{p-2}{p-1}|\omega(t)|^{p-1}. 
\end{equation*}
We now observe that
\begin{equation}\label{sphericlowerbound}
    \begin{aligned}
       \Phi_1(s) =& \left( C_1 + C_2\left|\frac{S_1(s)}{S_1'(s)}\right|^2\right) |\omega(s)|^{p-1} + C_3\left|\frac{S_1(s)}{S_1'(s)}\right| \left( \frac{|\omega'(s)|^{p-1}}{p-1}- |\omega(s)|^{p-2}|\omega'(s)|  \right) \\
           \geq& \left( C_1 + C_2\left|\frac{S_1(s)}{S_1'(s)}\right|^2\right) |\omega(s)|^{p-1} -  C_3\dfrac{p-2}{p-1}\left|\frac{S_1(s)}{S_1'(s)}\right||\omega(t)|^{p-1} \\
           =& \left( C_1 + C_2\left|\frac{S_1(s)}{S_1'(s)}\right|^2 - C_3\dfrac{p-2}{p-1}\left|\frac{S_1(s)}{S_1'(s)}\right| \right)|\omega(t)|^{p-1}.
    \end{aligned}
\end{equation}
The positivity of the last term in \eqref{sphericlowerbound} can be verified by an elementary computation. The idea is to show that the function lies above a quadratic barrier $Q$ on this interval, as illustrated in the figure below
 
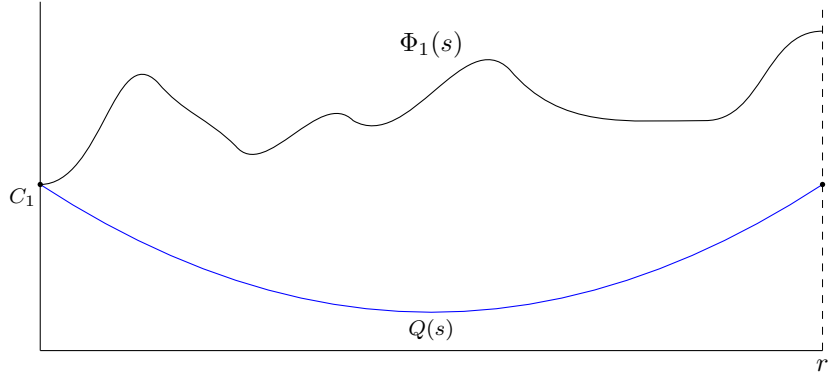
\begin{figure}[htb!]
    \centering
\begin{tikzpicture}[scale=1,>=stealth,x=5.2cm,y=1.7cm]
		\def\u{3.0}
		\def\r{2.0}
		\def\v{2.0}
		\def\t{2.0}
		\coordinate (p1) at (-\r,0.5);
		\coordinate (p2) at (0,\v);
		\coordinate (p3) at (\r,0.5);

		\path[draw,name path=l1] (-0,0) -- (2,0);
		\path[draw,name path=l2] (0,0) -- (0,0.91*\u);
		\path[draw,name path=l3,dashed,very thin] (2,0) -- (2,0.9*\u);
		\path[draw,domain=-0:2,blue,name path=para] plot ({\x},{0.3+pow(\x-1,2)});
		
		\path[name intersections={of=l2 and para}];
		\coordinate (p0) at (intersection-1);
		\fill (p0) circle (1pt);
		
		\path[name intersections={of=l3 and para}];
		\coordinate (p1) at (intersection-1);
		\fill (p1) circle (1pt);
		
		\node[below=-1] (t2) at (1,0.3) {\footnotesize{$Q(s)$}}; 
		\node[below] (t3) at (2,0) {$r$};
		\node[] (t4) at (1,2.4) {\normalsize{$\Phi_1(s)$}};
		\node[below left=-2] (t5) at (p0) {\footnotesize{$C_1$}};

		\draw[] (p0) to[out=0] (0.3,2.1) to[out=-50] (0.5,1.6)  to[out=-50] (0.8,1.8) to[out=-30] (1.2,2.2) to[out=-50,in=180] (1.7,1.8)  to[out= 0,in=180] (2,2.5);
	\end{tikzpicture}
\caption{Quadratic barrier for $\Phi_1$.}
\end{figure}

\noindent Consequently, $\Phi_1(t)>0$ for all $t\in(0,r)$, which ensures that $\Psi_1(t)>0$ on the same interval. Therefore, it follows from \eqref{identitysheprical} that, for all $t\in(0,r)$ with $r<\pi/2$, one has $\mathcal{V}_1'(t)\bigl(\omega(t)\bigr)^{p-1}-\mathcal{V}_1(t)|\omega'(t)|^{p-2}\omega'(t)\ge 0$, as required in \eqref{restricmodelshpere}.

\medskip
\noindent\textbf{Case 3: The hiperbolic model space ($c=-1$).} For the remaining case, we consider the model function \eqref{espaceformfunc} with $c=-1$. In this setting, inequality \eqref{inequality2} holds, and it suffices to prove that

\begin{equation}\label{restricmodelhiperbolic}
        \left( p+m-2 \right)\frac{S'_{-1}(t)}{S_{-1}(t)}|\omega'(t)|^{p-2}\omega'(t) + \lambda_{1,p}^{\mathcal{D}}(B(r))(\omega(t))^{p-1} \leq 0.
\end{equation}
A straightforward computation yields the following identity
\begin{equation}\label{modelhiperbolic1}
    \left( S_{-1}^{m-1}(t)|\omega'(t)|^{p-2}\omega'(t)\right)' + \lambda_{1,p}^{\mathcal{D}}(B(r))S_{-1}^{m-1}(t)(\omega(t))^{p-1} = 0.
\end{equation}
Define the function $\mathcal{V}_{-1}\colon (0,r)\to\mathbb{R}$ by
\begin{equation*}
    \mathcal{V}_{-1}(t) \coloneqq  \left( S_{-1}'(t) \right)^{-\frac{\lambda_{1,p}^{\mathcal{D}}(B(r))}{p+m-2}}.  
\end{equation*}
After some algebraic manipulations, we obtain the following equation
\begin{equation*}\label{modelhiperbolic2}
    \left(S_{-1}^{m-1}(t)\mathcal{V}_{-1}'(t) \right)' + \lambda_{1,p}^{\mathcal{D}}(B(r))S_{-1}^{m-1}(t)\left[ \frac{m}{p+m-2} - \frac{1}{p+m-2} \left( 1 + \dfrac{\lambda_{1,p}^{\mathcal{D}}(B(r))}{p+m-2} \right)\left|\frac{S_{-1}(t)}{S_{-1}'(t)}\right|^2  \right]\mathcal{V}_{-1}(t)=0.
\end{equation*}
We multiply the equation above by $-(\omega(t))^{p-1}$ and equation \eqref{modelhiperbolic1} by $\mathcal{V}_{-1}(t)$. Adding the resulting expressions and integrating from $0$ to $t$, we obtain the following
\begin{equation}\label{identityhiperbolic}
    \begin{aligned}
        S_{-1}^{m-1}(t)\left( \mathcal{V}_{-1}'(t)\left( \omega(t) \right)^{p-1} - \mathcal{V}_{-1}(t)|\omega'(t)|^{p-2}\omega'(t) \right) &= \int_{0}^{t} S_{-1}^{m-1}(s)\omega'(s) G(s)\mathcal{V}_{-1}'(s)ds \\
    &+ C_2(m,p) \int_{0}^{t}  S_{-1}^{m-1}(s)(\omega(s))^{p-1}\left|\frac{S_{-1}(s)}{S_{-1}'(s)}\right|^2\mathcal{V}_{-1}(s)ds \\
    &+C_4(m,p)\int_{0}^{t}  S_{-1}^{m-1}(s)(\omega(s))^{p-1}\mathcal{V}_{-1}(s)ds.
    \end{aligned}
\end{equation}
Here, $C_4(m,p) \coloneq \dfrac{p-2}{p+m-2}$, $C_2(m,p)$ is defined as in \eqref{CK}, and $G$ is given by \eqref{flateq}. In this setting, the sign of the right-hand side of \eqref{identityhiperbolic} is completely determined by the sign of the function $\Psi_{-1}\colon (0,r)\to\mathbb{R}$, defined by
\begin{equation*}
       \Psi_{-1}(t) \coloneq \int_{0}^{t}  S_1^{m-1}(s)  \underbrace{ \left[ \left( C_4 + C_2\left|\frac{S_{-1}(s)}{S_{-1}'(s)}\right|^2\right) |\omega(s)|^{p-1} + C_3\left|\frac{S_{-1}(s)}{S_{-1}'(s)}\right| |\omega'(s)| \mathcal{H}(s) \right]}_{\coloneq \Phi_{-1}(s)}\mathcal{V}_{-1}(s)ds. 
\end{equation*}
Given $C_3 = C_3(m,p)$ in the spherical case, we have that $\Psi_{-1}'(t) = S_1^{m-1}(t)\Phi_{-1}(t)\mathcal{V}_{-1}(t)$. Since the weight $S_1^{m-1}(t)\mathcal{V}_{-1}(t)$ is strictly positive for all $t \in (0,r)$, it follows that the qualitative behavior of the derivative $\Psi_{-1}'$ is entirely determined by the sign of the function $\Phi_{-1}$. Recall that $G(t) = (p-1)|\omega(t)|^{p-2} - |\omega'(t)|^{p-2}$. In this setting, for $p=2$ it is straightforward to verify that $\Psi_{-1}(t)>0$. Therefore, it remains to analyze the case $p\in(2,\infty)$. Recall that $\omega(0)=1$, $\omega'(0)=0$, and $\omega'(r)<0$. Consequently, we obtain that $\Phi_{-1}(0)=C_4(m,p)$ and that
\[\Phi_{-1}(r)= -C_3(m,p) \left| \frac{S_{-1}(r)}{S_{-1}'(r)} \right||\omega'(r)|^{p-1}<0.\]
Since $\omega \in \mathcal{C}^{1,\alpha}(B(r))$, it follows that the function $\Phi_{-1}$ is H\"older continuous. Consequently, $\Phi_{-1}(t)$ admits at least one real zero in $(0,r)$. Under these assumptions, for $c=-1$ we define $r_\star(c)$ as the real number determined as follows.
\begin{equation}
r_\star(-1) \coloneq \inf \left\{ t>0 \; ; \; \int_{0}^{t} S_1^{m-1}(s) \Phi_{-1}(s) \mathcal{V}_{-1}(s) ds \leq 0 \right\}.
\end{equation}
The function $\Phi_{-1}$ enjoys the same structural properties illustrated in Figure \eqref{fig:placeholder}, after adapting the construction to the hyperbolic case. Hence, for all $t \in (0,r_\star(-1))$, the right-hand side of \eqref{identityhiperbolic} is positive, which implies
\[
\mathcal{V}_{-1}'(t)\bigl(\omega(t)\bigr)^{p-1}
-\mathcal{V}_{-1}(t)\,|\omega'(t)|^{p-2}\omega'(t)
\ge 0.
\]
Consequently, the inequality in \eqref{restricmodelhiperbolic} holds on $(0,r_\star(-1))$. We now return to inequality \eqref{inequality2}. From Cases 1, 2, and 3, it follows that for every $t \in (0,r_\star(c))$, where $r_\star(c)$ is a constant depending only on the model space $\mathbb{M}^m(c)$,
\begin{equation*}
\left( \frac{S'_c(t)}{S_c(t)}\omega'(t) - \omega''(t) \right)  |\omega'(t)|^{p-2} |\omega(t)|^{1-p} \sin^2\! \alpha(x) \leq 0.
\end{equation*}
Thus, in the interval $(0,r_\star(c))$, using equation \eqref{pLaplacianoimersão2}, we obtain the following comparison inequality for the function $\psi$
\begin{equation}\label{modelcontrol}
    -\frac{\triangle_p \psi(x)}{(\psi(x))^{p-1}\cos^{p-2}{\left(\alpha(x)\right)}} \geq \lambda_{1,p}^{\mathcal{D}}(B(r_\star(c))).
\end{equation}
Recall that $\mathcal{F}=\varphi^{-1}(\operatorname{Cut}_N(p))$ and that $\mathcal{H}^{m-1}(\Omega\cap\mathcal{F})=0$. Therefore, the estimate \eqref{pbartabessamontenegro} can be applied to \eqref{modelcontrol} in $\Omega$. According to hypothesis we have, the extrinsic distance function $t=\operatorname{dist}_N(\varphi(\cdot),p)$ restricted to $\Omega$ is a uniform submersion, namely $|\nabla^{M} t|\ge k>0$ in $\Omega$. Since $|\nabla^{M} t|=\cos\alpha(x)$, it follows that $\cos\alpha(x)\ge k$ in $\Omega$, and consequently we obtain the following estimate.

\begin{equation*}
    \lambda_{1,p}(\Omega) \geq \inf_{x \in {\Omega \setminus \mathcal{F}}} \left\{ - \frac{\triangle_p \psi(x)}{\psi(x)^{p-1}} \right\} \geq k^{p-2} \cdot \lambda_{1,p}^{\mathcal{D}}(B(r)),
\end{equation*}
This guarantees the desired inequality in \eqref{ppBessaMontenegro}. We now establish the rigidity criterion required in Theorem \ref{thm:pminsub_ball}. Assume that $\Omega$ is bounded and that $\lambda_{1,p}(\Omega)=k^{p-2} \cdot \lambda_{1,p}^{\mathcal{D}} \bigl(B(r_\star(c))\bigr)$. Observe that $\varphi(\partial\Omega) \subseteq \partial B(r_\star(c))$. Since $\varphi\in \mathcal{C}^{\infty}(M)$, it follows that $\psi\in \mathcal{C}^{1,\alpha}(\Omega)$, and hence $\psi$ may be used as a test function in Theorem \ref{thm:pbarta}. Consequently, we obtain $\lambda_{1,p}(\Omega)=\triangle_p\psi/\psi^{p-1}$. However, in view of \eqref{pLaplacianoimersão2}, this identity forces $\cos{\alpha(x)}=k_0$ and therefore the following equality
\begin{equation*}
 \begin{aligned}
 \sum_{i=2}^m \left( \operatorname{Hess}_N(t)(e_i, e_i) - \frac{S'_c(t)}{S_c(t)}\right) |\omega' (t)|^{p-2} |\omega(t)|^{1-p} \omega' (t)&= 0, \\
 \left( \frac{S'_c(t)}{S_c(t)}\omega'(t) - \omega''(t) \right)  |\omega'(t)|^{p-2} |\omega(t)|^{1-p} \sin^2\! \alpha(x) &= 0, \\
 \left( \operatorname{Hess}_N(t) \left( \dfrac{\partial}{\partial \theta}, \dfrac{\partial}{\partial \theta} \right) - \frac{S'_c(t)}{S_c(t)}\right) |\omega' (t)|^{p-2}|\omega(t)|^{1-p} \omega'(t) \sin^2\! \alpha(x) &= 0,   
 \end{aligned}   
\end{equation*}
for all $t$ such that $\varphi(x)=\exp_{p_0}(t\xi)$. Therefore, $ \sin^2\! \alpha(x)=0$ and $\cos^2 \alpha(x)=1$ for all $x \in \Omega$ and  we have that $e_1(\varphi(x))= \partial/\partial t$. We now integrate the vector field $\partial/\partial t$. It follows that, in $N \cap \varphi(\Omega)$, there exists a minimizing geodesic joining $\varphi(x)$ to $p_0$. Consequently, $\Omega$ is a geodesic ball in $M$ of radius $r$, centered at $\varphi^{-1}(p_0)$. Since $\psi$ is a 
$p$-eigenfunction transplanted to $M$ from the model space $\mathbb{M}(c)$, it follows that
\[
    \triangle_p^{M} \omega(t) = \triangle_p^{\mathbb{M}(c)} \omega(t), \quad \text{for all} \ t = \operatorname{dist}(p_0,\varphi(q)) \ \text{with} \ q \in M.
\]
Finally, it suffices to rewrite the above identity in geodesic coordinates, which yields the following equation
\begin{equation*}
  \left( |\omega'(t)|^{p-2} \omega'(t) \right)' +  \frac{\sqrt{g(t,\xi)}\,'}{\sqrt{g(t,\xi)}}(t,\theta)|\omega'(t)|^{p-2} \omega'(t)   = \left( |\omega'(t)|^{p-2} \omega'(t) \right)' + (m-1) \frac{S_{c}'(t)}{S_{c}(t)} |\omega'(t)|^{p-2} \omega'(t).
\end{equation*}
Therefore, since $w'(t)=0$ if and only if $t=0$, it follows from the Bishop--Gromov Comparison Theorem that $\Omega = B_M(\varphi^{-1}(p_0), r)$ is isometric to the geodesic ball $B(r)$ in the model space $\mathbb{M}(c)$.
\end{proof}

\subsection{Proof of Corollary \ref{radiuscor}}
\textbf{Corollary \ref{radiuscor}}
\textit{Assume the hypotheses of Theorem \ref{thm:pminsub_ball}
in the Euclidean case ($c=0$).
Let $\varphi : M^m \hookrightarrow \mathbb{R}^n$ be a minimal immersion and let 
$\Omega$ be a connected component of 
$\varphi^{-1}(B_{\mathbb{R}^n}(p_0,r))$.
Then
\begin{equation}\label{lowerradiostone}
r \geq \left( k^{p-2} \cdot \frac{\lambda_{1,p}^{\mathcal D}(B_1^m)}{\lambda_{1,p}(\Omega)}\right)^{1/p},
\end{equation}
where $B_1^m$ denotes the unit ball in $\mathbb{R}^m$.}

\begin{proof}
Let $B_r^{m}$ denote the Euclidean ball of radius $r$ in $\mathbb{R}^{m}$. To prove the desired result, it suffices to show that
\begin{equation}\label{scaling}
    \lambda_{1,p}^{\mathcal{D}}(B^{m}_r) = \frac{1}{r^p}\lambda_{1,p}^{\mathcal D}(B_1^m).
\end{equation}
Indeed, if \eqref{scaling} holds, then Theorem \ref{thm:pminsub_ball}
yields the following lower bound for the $p$-fundamental tone
\begin{equation}
    \lambda_{1,p}(\Omega) \geq k^{p-2} \cdot \lambda_{1,p}^{\mathcal{D}}(B^{m}_r) = k^{p-2} \cdot \frac{1}{r^p}\lambda_{1,p}^{\mathcal D}(B_1^m).
\end{equation}
Thus establishing the desired inequality in \eqref{lowerradiostone}.
It remains to prove \eqref{scaling}. To this end, we use the variational
characterization given in \eqref{eq:Rayleigh quotient}. Consequently,
\begin{equation}
      \lambda_{1,p}^{\mathcal{D}}(B^{m}_r)
    = \inf \left\{ \frac{\displaystyle \int_{B^{m}_r} |\nabla u |^p dx} {\displaystyle \int_{B^{m}_r} |u|^p dx}; \, u \in \mathrm{W}^{1,p}_0(\Omega), \, \varphi \not\equiv 0
    \right\}.
\end{equation}
Now consider the change of variables $x=ry$ and define $v(y)=u(ry)$. Then
$\nabla_y v(y)= r\, \nabla_x u(ry)$, and, since $dx = r^{m}dy$, it follows that 
\begin{equation}\label{eqcor}
     \frac{\displaystyle \int_{B^{m}_r} |\nabla u |^p dx} {\displaystyle \int_{B^{m}_r} |u|^p dx} =  r^{-p} \cdot\frac{\displaystyle \int_{B^{m}_1} |\nabla v |^p dy} {\displaystyle \int_{B^{m}_1} |v|^p dy}.
\end{equation}
Finally, note that 
\[
T\colon W^{1,p}_0(B_r^{m}) \to W^{1,p}_0(B_1^{m}), 
\qquad (Tu)(y) = u(ry),
\]
is a bijective mapping. Hence, taking the infimum over 
$W^{1,p}_0(B_r^{m})$ in \eqref{eqcor} gives the equality in \eqref{scaling}.
\end{proof}

\subsection{Proof of Theorem \ref{thmlowerptone}}

We now establish a quantitative lower bound for the first $p$-eigenvalue of extrinsic domains in submanifolds immersed in a product space $N \times \mathbb{R}$.  The estimate reflects the interplay between the radial curvature of the ambient manifold $N$ and the mean curvature of the immersion. More precisely, under an upper radial curvature bound on $N$ and a local control on the mean curvature, we obtain an explicit positive lower bound for the $p$-fundamental tone of connected components of extrinsic geodesic balls. The result can be interpreted as a comparison theorem in which the quantity $(m-2)(S'_c/S_c)(r) - h(x_0,r)$ measures the competition between the ambient curvature and the mean curvature of the immersion.
\medskip

\noindent \textbf{Theorem 1.7}
Let $\varphi\colon M \hookrightarrow N \times \mathbb{R}$ be a complete immersed $m$-submanifold with locally bounded mean curvature. Assume that $N$ has radial sectional curvature bounded above by $K_N(\partial_t,v) \leq c$ along the geodesics issuing from a point $x_0\in N$. Let $\Omega(r)\subset \varphi^{-1}\left(B_N(x_0,r)\times\mathbb{R}\right)$ be a connected component and suppose that $r \leq \min \left\{\operatorname{inj}_N(x_0), \pi/2\sqrt{c} \right\}$. Additionally, suppose that $r$ is chosen
\begin{enumerate}
\item If $|h(x_0,r)|< \Lambda < \infty$ let $r \leq (S_c / S'_c)(\Lambda/(m-2))$;
\item If $\displaystyle \lim_{r \to \infty}h(x_0,r) = \infty$ let $r \leq (S_c / S'_c)(h(x_0,r)/(m-2))$, where $r_0$ is so that $(m-2)(S'_c/S_c)-h(x_0,r_0)=0.$
\end{enumerate}
Then, for every $p>1$, whenever either $1$ or $2$ holds, it follows that
\begin{equation}\label{lowestimateptone}
 \lambda_{1,p}(\Omega(r)) \geq    \frac{1}{p^{p}}
\left((m-2)\frac{S'_c}{S_c}(r)-h(x_0,r)\right)^{p} > 0.
\end{equation}

\begin{proof}
Let $\varphi \colon M \hookrightarrow N \times \mathbb{R}$ be a complete immersed $m$-dimensional submanifold with locally bounded mean curvature, and assume that $N$ has radial sectional curvature bounded above by $K_N(\partial_t,v) \leq c$ along the geodesics issuing from a point $x_0 \in N$. Let $\Omega(r)$ be a connected component of $\Omega(r) \coloneqq \varphi^{-1}\bigl(B_N(x_0,r) \times \mathbb{R}\bigr)$. Define the function $\widetilde{t} \colon N \times \mathbb{R} \to \mathbb{R}$ by $\widetilde{t}(x,s) \coloneqq \operatorname{dist}_N(x,x_0)$, and set $\widetilde{\varphi} \coloneqq \widetilde{t} \circ \varphi \colon \Omega(r) \to \mathbb{R}$ and $X \coloneqq \nabla \widetilde{\varphi}$. The goal here will be to find $r \leq \min \left\{\operatorname{inj}_N(x_0), \pi/2\sqrt{c} \right\}$, with the convention that ${\pi}/{2\sqrt{c}} = +\infty$ if $c \leq 0$, and that $\inf_{\Omega(r)} \operatorname{div} X > 0$. Under these assumptions, we may apply Theorem \ref{LimaSantosMontenegro} to the vector field $X$ in order to obtain the following lower bound for the $p$-fundamental tone of $\Omega(r)$.
\begin{equation}\label{ptoneestimateee}
     \lambda_{1,p}(\Omega(r)) \geq \left( \frac{\displaystyle \inf_{\Omega(r)}{\operatorname{div}(X)}}{p  \Vert X \Vert_{\infty}} \right)^p. 
\end{equation}

Let us initially consider an orthonormal basis 
$\{\nabla t, \partial/\partial \theta_1, \ldots, \partial/\partial \theta_{n-1}, \partial/\partial s\}$ 
of $T_{(q,s)}(N \times \mathbb{R})$, where 
$\{\nabla t, \partial/\partial \theta_1, \ldots, \partial/\partial \theta_{n-1}\}$ 
is an orthonormal basis of $T_q N$ in polar coordinates centered at $x_0$. 
Let $\{e_1, \ldots, e_m\}$ be an orthonormal basis of $T_x\Omega$, and write
\begin{equation}
    e_i = \alpha_i \cdot \nabla t_N + \beta_i \cdot \partial/\partial \theta_s + \sum_{j=1}^{n-1}\gamma_i^{j}\partial/\partial \theta_j.
\end{equation}
Where $\alpha_i$, $\beta_i$, and $\gamma_i^{\,j}$ are real numbers satisfying $\alpha_i^{2} + \beta_i^{2} + \sum_{j=1}^{n-1} (\gamma_i^{\,j})^{2} = 1$ for $i = 1,\ldots,m$. Note that $\operatorname{div}(X) = \triangle \widetilde{\varphi}$. Moreover, using the Jorge--Koutroufiotis formula \cite{JorgeKoutrofiotis1981} and identifying the orthonormal basis $\{e_1,\ldots,e_m\}$ of $T_x M$ with $\{d\varphi(e_1),\ldots,d\varphi(e_m)\}$ in $T_{\varphi(x)}(N \times \mathbb{R})$, it follows that
\begin{equation}\label{JKCCC}
    \triangle\Tilde{\varphi}(x) = \sum_{i=1}^m \operatorname{Hess} \, \Tilde{t}(\varphi(x))(e_i,e_i)
   + \left\langle \nabla \Tilde{t} , \sum_{i=1}^m \mathrm{I}\mathrm{I}(e_i,e_i)  \right\rangle .
\end{equation}
Thus, using the estimates given in equations (3.8) and (3.9) in \cite{Bessa-Silvana}, we obtain that
\begin{equation*}
    \sum_{i=1}^{m}\operatorname{Hess} \, \Tilde{t}(e_i,e_i) \geq \sum_{i=1}^{m}(1 -\alpha_i^2 -\beta_i^2)\frac{S'_c}{S_c}(r) \quad \text{and} \quad  \left\langle \nabla \Tilde{t} , \sum_{i=1}^m \mathrm{I}\mathrm{I}(e_i,e_i)  \right\rangle \geq -h(x_0,r)\sqrt{1 - \sum_{i=1}^{m}\alpha_i}.
\end{equation*}
Using the identity $\alpha_i^{2} + \beta_i^{2} + \sum_{j=1}^{n-1} (\gamma_i^{\,j})^{2} = 1$ in the first inequality above, adding the resulting expression to the right-hand side, and substituting it into \eqref{JKCCC}, we obtain the following estimate
\begin{equation*}
    \triangle\Tilde{\varphi} \geq (m-2)\dfrac{S'_c}{S_c}(r) - h(x_0,r)>0.
\end{equation*}
Therefore, for the vector field $X = \nabla \widetilde{\varphi}$, in both cases (1) and (2) considered above, the inequality \eqref{ptoneestimateee} (see the detailed argument in \cite[p. 8]{Bessa-Silvana}) yields the following estimate
\begin{equation*}
 \lambda_{1,p}(\Omega(r)) \geq    \frac{1}{p^{p}}
\left((m-2)\frac{S'_c}{S_c}(r)-h(x_0,r)\right)^{p} > 0.
\end{equation*}
as initially required.
\end{proof}

\subsection{Proofs of the \textit{p}-stability criteria}
\textbf{Corollary \ref{p-stability}} \textit{Assume the hypotheses of Theorem \ref{thm:pminsub_ball}.
Let $\varphi : M^m \hookrightarrow N^n$ be a minimal immersion and let 
$\Omega$ be a connected component of 
$\varphi^{-1}(B_N(p_0,r))$.
Fix $2 \leq p < \infty$ and set $\mathcal V(x) = \|A(x)\|^p$ where $A$ denotes the second fundamental form of the immersion. If
\begin{equation}\label{p-st}
\sup_{x \in \Omega} \|A(x)\|^p \leq k^{p-2} \cdot \lambda_{1,p}^{\mathcal D}(B(r)).
\end{equation}
Then $\Omega$ is $p$-stable with respect to the potential $\|A\|^p$.}

\begin{proof}
Let $\Omega \subset M$ be as in Theorem \ref{thm:pminsub_ball}.
For $2 \le p < \infty$. Consider the $p$-stability functional, given $u \in \mathcal{C}_0^\infty(\Omega)$
\begin{equation}\label{s1}
    Q_p(u) \coloneq \int_{\Omega} \Big( |\nabla u|^p - \|A\|^p |u|^p \Big).
\end{equation}
By the variational characterization of the first Dirichlet eigenvalue
of the $p$-Laplacian given in \eqref{eq:Rayleigh quotient}, we have
\[
 \lambda_{1,p}^{\mathcal{D}}(\Omega)
    = \inf \left\{ \frac{\displaystyle \int_{\Omega} |\nabla u |^p d\mu} {\displaystyle \int_{\Omega} |u|^p d\mu}; \, u \in \mathrm{W}^{1,p}_0(\Omega), \, \varphi \not\equiv 0
    \right\}.
\]
Therefore, by density argument (see \cite[Corollary A.4]{le}), for all  $u \in \mathcal{C}_0^\infty(\Omega)$
\begin{equation}\label{s2}
\int_{\Omega} |\nabla u|^p \geq \lambda_{1,p}(\Omega) \int_{\Omega} |u|^p. 
\end{equation}
Combining \eqref{s1} and \eqref{s2}, we obtain the estimate
\[
Q_p(u) \geq \big( \lambda_{1,p}(\Omega) - \sup_{\Omega} \|A\|^p\big)
\int_{\Omega} |u|^p.
\]
According to the hypotheses of Theorem \ref{thm:pminsub_ball}, we have $\lambda_{1,p}(\Omega) \geq k^{p-2} \cdot \lambda_{1,p}^{\mathcal D}(B(r))$. By the assumption given in \eqref{p-st}, $\sup_{x \in \Omega} \|A(x)\|^p \leq k^{p-2} \cdot \lambda_{1,p}^{\mathcal D}(B(r))$, and hence $\lambda_{1,p}(\Omega) - \sup_{\Omega} \|A\|^p \ge 0$. It follows that $Q_p(u) \ge 0$ for every $u \in C_0^\infty(\Omega)$, which proves the $p$-stability of $\Omega$ with respect to the potential $\|A\|^p$. 
\end{proof}
\medskip
\noindent
 \textbf{Corollary 1.5} \textit{Assume the hypotheses of Theorem \ref{thmlowerptone}. Let $\varphi \colon M \hookrightarrow \mathbb{R}^{n}$ be a minimal immersion into Euclidean space. Suppose that the second fundamental form $A$ in $\Omega_r$ satisfies
\begin{equation*}
    \|A\| \leq \frac{m-1}{pr}.
\end{equation*}
Then $\Omega_r$ is $p$-stable with respect to the potential $\|A\|^{p}$.}
\begin{proof}
Let $\Omega_r \subset M$ be as in Theorem \ref{thmlowerptone}.
For $2 \le p < \infty$. Consider the $p$-stability functional, given $u \in \mathcal{C}_0^\infty(\Omega)$
\begin{equation*}
    Q_p(u) \coloneq \int_{\Omega} \Big( |\nabla u|^p - \|A\|^p |u|^p \Big).
\end{equation*} 
It follows from equations \eqref{s1} and \eqref{s2} in Corollary \ref{p-stability} that the following inequality holds
\[
Q_p(u) \geq \big( \lambda_{1,p}(\Omega) - \sup_{\Omega} \|A\|^p\big)
\int_{\Omega} |u|^p.
\]
According to the hypotheses of Corollary \ref{corpstaby}, we have
\begin{equation*}
     \|A\|^{p} \leq \left( \frac{m-1}{p} \right)^{p} \cdot \frac{1}{r^p}.
\end{equation*}
By \eqref{lowestimateptone} in Theorem \ref{thmlowerptone}, we have $\lambda_{1,p}(\Omega) \ge \sup_{\Omega} \|A\|^{p}$. Hence, $\Omega_r$ is $p$-stable with respect to the potential $\|A\|^{p}$.
\end{proof}

\subsection{Proof of Theorem \ref{Thm:pKazdame}}

The result below can be interpreted within the framework of criticality theory for quasilinear elliptic operators. Indeed, the existence of a boundary blow-up solution to \eqref{intro:semilineareq} is intimately related to the sign of the principal Dirichlet eigenvalue of the operator $L_\Psi = -\triangle_p - \Psi$. Our theorem shows that the existence of such a solution forces the potential to lie below the principal eigenvalue, while the limiting case corresponds to a critical regime in which rigidity occurs and the potential must be constant. In this sense, the boundary blow-up problem detects the critical threshold of the operator.
\medskip

\noindent
\textbf{Theorem 1.8}
\textit{Let $M$ be a bounded Riemannian manifold with smooth boundary, and let 
$\Psi \in \mathcal{C}^0(\overline{M})$. Consider the problem
\begin{equation}\label{intro:semilineareq}
\begin{cases}
\triangle_p u - (p-1)|\nabla u|^p = \Psi & \text{in } M,\\
u = +\infty & \text{on } \partial M.
\end{cases}
\end{equation}
If \eqref{intro:semilineareq} admits a $\mathcal{C}^{1,\alpha}_{{\rm loc}}(M)$ solution, then $\inf_M \Psi \leq \lambda_{1,p}^{\mathcal{D}}(M)$. Moreover, if $\inf_M \Psi = \lambda_{1,p}^{\mathcal{D}}(M)$, then $\Psi \equiv \lambda_{1,p}^{\mathcal{D}}(M)$. Finally, if $\Psi \equiv \lambda$ is constant, then problem \eqref{intro:semilineareq} admits a solution if and only if $\lambda = \lambda_{1,p}^{\mathcal{D}}(M)$.}

\begin{proof} Let $(M,g)$ be a bounded Riemannian manifold with smooth and nonempty boundary.
We consider the following quasilinear problem with potential
\begin{equation}\label{semilineareq1}
\begin{cases}
\triangle_p \varphi + \Psi \varphi^{p-1} = 0 \quad & \text{in} \ \ M,\\
\varphi =0 & \text{on } \partial M,
\end{cases}
\end{equation}
where $p > 1$ and $\Psi \in \mathcal{C}^0(\overline{M})$. We begin by observing that problem \eqref{intro:semilineareq} admits a positive solution $v \in \mathcal{C}^{1,\alpha}_{\mathrm{loc}}(M)$ if and only if problem \eqref{semilineareq1} admits a positive solution $\varphi \in \mathcal{C}^{1,\alpha}(M)\cap \mathcal{C}^0(\overline{M})$. Following the approach of Kazdan--Kramer \cite{KazdanKramer}, we introduce the change of variables $v=-\log \varphi$. With this transformation, we obtain
\[
    \nabla v = -\frac{\nabla \varphi}{\varphi},
    \qquad 
    | \nabla v |^p = \frac{ |\nabla \varphi |^p}{|\varphi|^p}.
\]
Hence, we obtain the identity
    \begin{equation*}
    \triangle_p v 
    = \operatorname{div}( | \nabla v |^{p-2}\nabla v)
    = \operatorname{div}\!\left(
        -\frac{ |\nabla \varphi |^{p-2}\nabla \varphi}{\varphi^{p-1}}
      \right).
\end{equation*}
Thus, it follows that (see \cite[Remark 2.1]{LimaMontenegroSantos2010})
\begin{equation}\label{kazdamfonte}
    - \operatorname{div}\!\left( \frac{ |\nabla \varphi |^{p-2}\nabla \varphi}{\varphi^{p-1}}
      \right) =  (p-1) \dfrac{|\nabla \varphi |^p}{|\varphi|^p} - \frac{\triangle_p \varphi}{\varphi^{p-1}}.
\end{equation}
Therefore, we conclude that
\begin{equation*}
    \triangle_p v - (p-1)| \nabla v |^p 
    =  - \frac{\triangle_p \varphi}{\varphi^{p-1}}.
\end{equation*}
Moreover, since $\varphi \in \mathcal{C}^{1,\alpha}(M)$ and $\varphi|_{\partial M} \equiv 0$, it follows that $v \in \mathcal{C}^{1,\alpha}_{\mathrm{loc}}(M)$. Furthermore, we have that $\lim_{x \to \partial M} v(x) = +\infty$. Consequently, $v$ solves the $p$-Laplacian equation with gradient absorption term
\begin{equation}\label{intro:semilineareq2}
\begin{cases}
\triangle_p v - (p-1)|\nabla v|^p = \Psi & \text{in } M,\\
v = +\infty & \text{on } \partial M.
\end{cases}
\end{equation}
Thus, \eqref{intro:semilineareq2} is equivalent to \eqref{semilineareq1} under the transformation $\varphi = e^{-v}$. Since $\varphi \in \mathcal{C}^{1,\alpha}(M)$, $\triangle_p \varphi \in \mathcal{C}^{0}(M)$, and $\varphi>0$ in $M$, it follows that $\varphi$ is an admissible test function in Barta's inequality (see Theorem \ref{thm:pbarta}). Consequently, we obtain the following inequality
\begin{equation}\label{cadeia1}
\begin{aligned}
    \lambda_{1,p}^{\mathcal{D}}(M)  &\geq \inf_{M} \left\{ -\frac{\triangle_p \varphi}{\varphi^{p-1}}  \right\} \\
     &= \inf_{M} \left( \triangle_p v - (p-1)| \nabla v |^p \right) \\
    &= \inf_{M} \Psi.
\end{aligned}
\end{equation}
Since $\varphi$ is a positive solution of \eqref{semilineareq1}, it follows from \eqref{eqweak} that for every test function $\eta \in \mathcal{C}_0^{\infty}(M)$ we have
\begin{equation*}
\int_{M} \left\langle |\grad\varphi |^{p-2} \grad \varphi, \grad \eta \right\rangle d\mu =  \int_{M} \Psi(x) | \varphi |^{p-1} \eta d\mu.  
\end{equation*}
Choosing $\varphi = \eta$ and using the fact that $\Psi \in \mathcal{C}^0(\overline{M})$ we obtain the following inequality
\begin{equation*}
\int_{M}  |\grad\varphi |^{p} d\mu =  \int_{M} \Psi(x) | \varphi |^{p} d\mu \leq \sup_{x \in \overline{M}}  \Psi(x) \int_{M}  | \varphi |^{p} d\mu.
\end{equation*}
Therefore, since $\varphi \in W^{1,p}_0(M)$ we can use the variational characterization given in \eqref{eq:Rayleigh quotient} of the $p$-fundamental tone to obtain
\begin{equation}\label{Rayleigh quotient}
    \lambda_{1,p}^{\mathcal{D}}(M)
    \leq \frac{\displaystyle \int_{M} |\nabla \varphi  |^p d\mu} {\displaystyle \int_{M} |\varphi|^p d\mu} \leq \sup_{x \in \overline{M}}  \Psi(x).
\end{equation}
Finally, suppose that $\inf_{M} \Psi = \lambda_{1,p}^{\mathcal{D}}(M)$. Then equality must occur in \eqref{cadeia1}
\begin{equation*}
      \lambda_{1,p}^{\mathcal{D}}(M) = \inf_{M} \left\{ -\frac{\triangle_p \varphi}{\varphi^{p-1}}  \right\}.
\end{equation*}
In this case, it follows from Theorem \ref{thm:pbarta} that $\varphi$ is a scalar multiple of the first Dirichlet $p$-eigenfunction. Since $\varphi|_{\partial M}=0$ and in view of the $(p-1)$-homogeneity of the $p$-Laplacian, we conclude that $\varphi$ satisfies the following equation
\begin{equation}\label{varphi}
\begin{cases}
\triangle_p \varphi = -\lambda_{1,p}^{\mathcal{D}}(M)|\varphi|^{p-1} 
& \text{in} \ M,\\
\varphi = 0 
& \text{on} \ \partial M. 
\end{cases}
\end{equation}
Therefore, $\varphi$ satisfies both \eqref{semilineareq1} and \eqref{varphi}. It follows that, for every $x \in M$, one has $\Psi(x) - \lambda_{1,p}^{\mathcal{D}}(M) = 0$. By continuity, we conclude that $\Psi(x) = \lambda_{1,p}^{\mathcal{D}}(M)$ on $\overline{M}$. Finally, suppose that there exists a constant $\lambda$ such that $\Psi(x)=\lambda$ for all $x \in M$. If \eqref{semilineareq1} admits a solution, then it follows from \eqref{cadeia1} and \eqref{Rayleigh quotient} that
\[
\inf_M \Psi \le \lambda_{1,p}^{\mathcal{D}}(M) \le \sup_M \Psi.
\]
Hence, $\lambda = \lambda_{1,p}^{\mathcal{D}}(M)$. The converse implication is immediate.
\end{proof}

\section*{Acknowledgements}

The author would like to express his sincere gratitude to G. Pacelli Bessa for insightful discussions that significantly enriched this work. He also thanks L. Petrola Jorge for valuable notes concerning the $p$-Laplacian. In addition, the author gratefully acknowledges Vicent Gimeno and Vicent Palmer for their warm hospitality at the Universitat Jaume I during the doctoral research stay. This work was partially supported by CAPES--Brazil under Grant No. 88881.126989/2025-01.

\end{document}